\documentclass[12pt]{amsart}

\usepackage{amsmath}
\usepackage{mathtools}
\usepackage{amssymb}
\usepackage{amsthm}
\usepackage{amscd}
\usepackage{mathrsfs}
\usepackage{dsfont}
\usepackage{bbm}
\usepackage{enumitem}
\usepackage{float}
\usepackage[final]{showkeys}
\usepackage[colorlinks,linkcolor=red,anchorcolor=blue,citecolor=blue]{hyperref}

\newtheorem{thm}{Theorem}[section]
\newtheorem{cor}[thm]{Corollary}
\newtheorem{prop}[thm]{Proposition}
\newtheorem{lem}[thm]{Lemma}

\theoremstyle{definition}

\newtheorem{ex}[thm]{Example}
\theoremstyle{remark}
\newtheorem{rem}[thm]{Remark}
\numberwithin{equation}{section}

\setlist[enumerate]{label=$({\arabic*})$}

\newcommand{\abs}[1]{\left\lvert#1\right\rvert}
\newcommand{\norm}[1]{\left\lVert#1\right\rVert}
\newcommand{\e}{\varepsilon}
\newcommand{\E}{\mathbb{E}}
\newcommand{\R}{\mathbb{R}}

\newcommand{\cE}{\mathcal{E}}
\newcommand{\cF}{\mathcal{F}}
\newcommand{\cG}{\mathcal{G}}
\newcommand{\cH}{\mathcal{H}}

\newcommand{\cK}{\mathcal{K}}
\newcommand{\cM}{\mathcal{M}}
\newcommand{\cN}{\mathcal{N}}
\newcommand{\cO}{\mathcal{O}}

\newcommand{\cU}{\mathcal{U}}

\DeclareMathOperator{\Ent}{Ent}
\DeclareMathOperator{\diag}{diag}
\DeclareMathOperator{\MA}{MA}
\DeclareMathOperator{\PSH}{PSH}
\DeclareMathOperator{\Ric}{Ric}
\DeclareMathOperator{\sgn}{sgn}
\DeclareMathOperator{\supp}{Supp}
\DeclareMathOperator{\tr}{tr}
\DeclareMathOperator{\TV}{TV}
\DeclareMathOperator{\vol}{vol}

\newcommand{\ddc}{dd^c}

\newcommand{\underalign}[2]{\quad\underset{\mathclap{\strut #1}}{#2}\quad}

\title{Metric entropy of K\"ahler potentials}

\author{Tomoyuki HISAMOTO}
\address{Graduate School of Mathematics, Nagoya University, Nagoya, Japan}
\email{{\tt hisamoto@math.nagoya-u.ac.jp}}

\date{September 23, 2026}


\begin{document}

\begin{abstract}
We prove sharp metric entropy estimates for spaces of K\"ahler potentials. In complex dimension $n$, normalized potentials have Kolmogorov entropy of order $\e^{-n}$ in the background $L^1$ metric. On a polarized manifold, a relative-entropy sublevel has the same order in the Mabuchi--Darvas $d_1$ metric, including its full finite-energy closure. The upper bound is $C_X(1+B)^{n+1}\e^{-n}$ for entropy budget $B$. For toric potentials, the sharp exponent is $n/2$.
\end{abstract}
\maketitle
\begingroup
\renewcommand{\baselinestretch}{0.94}\selectfont
\tableofcontents
\endgroup

\section{Introduction}
The study of canonical K\"ahler metrics leads naturally to the
geometry of the space of K\"ahler metrics in a fixed cohomology class. 
Mabuchi's Riemannian structure \cite{Mab87} and the finite-energy
metric completions developed by Darvas
\cite{Dar15} provide a framework for studying this 
infinite-dimensional space through its intrinsic distances. 
For the Mabuchi--Darvas distance $d_1$, the compactness of normalized
relative-entropy sublevels follows from the pluripotential theory
of Berman, Boucksom, Eyssidieux, Guedj, and Zeriahi
\cite[Theorem~2.17 and Corollary~2.19]{BBEGZ19}. 
It plays a crucial role in the variational approach to canonical Kähler metrics. 
This qualitative compactness raises a quantitative question: 
how many metric balls of radius $\e$ are needed to cover
such a sublevel as $\e$ tends to zero?

For a metric space $(S,d)$ let $N_\e(S,d)$ be the least number of open balls of radius $\e$, with centers in $S$, covering $S$, and set
\begin{equation*}
 H_\e(S,d)=\log N_\e(S,d).
\end{equation*}
All logarithms are natural. This is the metric entropy, systematically developed by Kolmogorov and Tikhomirov \cite{KT59}. Its classical models include Sobolev classes \cite{BirSol67}, convex bodies \cite{Bro76}, and convex functions \cite{GW17}. Entropy questions for bounded holomorphic functions were studied by Babenko \cite{Bab58} and Erokhin \cite{Ero58}, and subsequently related to pluripotential theory by Zakharyuta \cite{Zak94} and Nivoche \cite{Niv04}. Further developments include \cite{BN22,Fin26}. Here the objects are real potentials constrained by complex Hessian positivity, and the intrinsic metric depends on their Monge--Amp\`ere measures.

\subsection*{Three entropy estimates}
Let $(X,\omega)$ be a compact connected K\"ahler manifold of complex dimension $n\geq1$. Write $\omega_u=\omega+\ddc u$, $V=\int_X\omega^n$, and $\nu=V^{-1}\omega^n$. Put
\begin{equation*}
 \cH_\omega=\{u\in C^\infty(X,\R):\omega_u>0,\ \sup_Xu=0\}.
\end{equation*}
Let $\PSH_0(X,\omega)$ be the upper semicontinuous $\omega$-plurisubharmonic potentials with supremum zero, where $\omega$-plurisubharmonic means $\omega+\ddc u\geq0$ as a current. The first natural 
candidate 
in our mind 
is the 
background $L^1$ metric  
\begin{equation*}
 d_{L^1}(u,v)=\int_X\abs{u-v}d\nu.
\end{equation*}
It is well-known that the space $\PSH_0(X,\omega)$ is compact in 
$L^1(\nu)$.

\begin{thm}
\label{thm-l1-main}
There are constants $c_X,C_X,\e_X>0$, depending only on $(X,\omega)$, such that
\begin{equation}
 c_X\e^{-n}\leq H_\e(\cH_\omega,d_{L^1})
 \leq C_X\e^{-n},\qquad 0<\e<\e_X.
 \label{eq-l1-main}
\end{equation}
Moreover,
\begin{equation}
 \overline{\cH_\omega}^{\,L^1}=\PSH_0(X,\omega),
 \label{eq-l1-closure}
\end{equation}
and the same two-sided estimate holds for this closure, after changing the constants.
\end{thm}

Our principal result concerns the Mabuchi--Darvas distance $d_1$.  
For probability measures on a common space, define relative entropy by
\begin{equation*}
 \Ent(\mu\mid\nu)=
 \begin{cases}\displaystyle\int\log(d\mu/d\nu)d\mu,&\mu\ll\nu,\\
 +\infty,&\text{otherwise},\end{cases}
\end{equation*}
with $0\log0=0$. Let $\cE^1(X,\omega)$ be the full Monge--Amp\`ere mass class with finite Monge--Amp\`ere energy; its precise definition and the non-pluripolar measure $\MA(u)$ are recalled in Section~\ref{sec-preliminaries}. The probability-normalized Mabuchi--Darvas metric $d_1$ is the completion of the path-length metric for the norm $\int
\abs{\dot u}\MA(u)$ of the tangent $\dot{u}$. For smooth $u$, $\MA(u)=V^{-1}\omega_u^n$. Define
\begin{equation*}
 \cK_B^{\cH}=\{u\in\cH_\omega:\Ent(\MA(u)\mid\nu)\leq B\}.
\end{equation*}
We denote by $\cK_B$ the corresponding class of all $u\in\cE^1$ with $\sup_Xu=0$, without a smoothness assumption.

\begin{thm}
\label{thm-d1-main}
Assume $\omega=c_1(L,h_0)$ for a positive Hermitian line bundle on $X$. There is a geometric constant $C_X>0$ such that, for every $B\geq0$,
\begin{equation}
 H_\e(\cK_B,d_1)
 \leq C_X(1+B)^{n+1}\e^{-n},\qquad 0<\e<\tfrac12.
 \label{eq-d1-upper-general}
\end{equation}
There are geometric constants $c_X,\e_X>0$ such that for every $B>0$,
\begin{equation}
 H_\e(\cK_B,d_1)
 \geq c_X\min\{1,B^{n/2}\}\e^{-n}
 \label{eq-d1-lower-main}
\end{equation}
whenever
\begin{equation}
 0<\e<\e_X\min\{1,\sqrt B\}.
 \label{eq-d1-lower-range}
\end{equation}
The same upper and lower estimates hold with $\cK_B^{\cH}$ in place of $\cK_B$, with internal centers in either case. In particular, for each fixed $B>0$ both entropies have order $\e^{-n}$. The lower bound does not require polarization.
\end{thm}

For $B\geq1$ the upper bound can be written $C_XB^{n+1}\e^{-n}$. We do not claim optimal dependence on $B$: at $B=0$ the normalized class is a singleton.

Finally, suppose that $X$ and the reference metric are toric. Let $P\subset\R^n$ be the Delzant moment polytope, $\lambda_P$ its normalized Lebesgue measure, and $u_0$ the reference symplectic potential. In this part we fix additive constants by requiring the symplectic potential to have zero $\lambda_P$-mean, rather than by supremum normalization on $X$. Write $\cK_B^{\rm tor}$ for the resulting torus-invariant finite-energy potentials with relative entropy at most $B$ with respect to $V^{-1}\omega^n$. 

\begin{thm}
\label{thm-toric-main}
There are constants $c, C, \e_0>0$, depending only on the reference toric metric, such that for every $B>0$,
\begin{equation}
 H_\e(\cK_B^{\rm tor},d_1)
 \leq C\left(\frac{1+B}{\e}\right)^{n/2},\qquad\e>0,
 \label{eq-toric-upper-intro}
\end{equation}
and
\begin{equation}
 H_\e(\cK_B^{\rm tor},d_1)
 \geq c\left(\frac{\min\{1,\sqrt B\}}{\e}\right)^{n/2}
 \label{eq-toric-lower-intro}
\end{equation}
for any $0<\e\leq \e_0\min\{1,\sqrt B\}$. Equivalently, these estimates hold in $L^1(\lambda_P)$ for the convex class
\begin{equation*}
 \left\{u\in L^1(P): u\text{ convex},\ \int_Pu\,d\lambda_P=0,
 \ \Ent(((\nabla u_0)^{-1}\circ\nabla u)_*\lambda_P\mid\lambda_P)\leq B\right\}.
\end{equation*}
Here gradients are defined almost everywhere in the interior of $P$. Thus the sharp exponent under toric symmetry is $n/2$.
\end{thm}

\subsection*{Background and structure of the proofs}
For Theorem~\ref{thm-l1-main}, positivity bounds the distributional Laplacian as a signed measure. 
Following the multiscale approximation methods of Birman--Solomyak \cite{BirSol67} and DeVore \cite{Dev98}, we approximate the measure-valued Laplacian by discrete measures that preserve local moments. 
The decisive analytic estimate is Lemma~\ref{lem-green-cell}: matching coordinate moments through degree two leaves a Green potential of $L^1$ size $O(h^2)$. Disjoint local bumps and a constant-weight Gilbert--Varshamov code \cite{Gil52,Var57} give the reverse inequality. 
For the $d_1$ lower bound, we use the same local bumps with amplitude $a$. The relative entropy is quadratic to leading order in $a$, with second derivative at $a=0$ given by the Fisher information of the resulting one-parameter family of Monge--Amp\`ere measures. Their pairwise $d_1$-separation is bounded below by a quantity proportional to $a$.

The upper bound in Theorem~\ref{thm-d1-main} 
is the core of this paper and consists of three parts. The first part starts from the quantized energy comparisons of Berman--Freixas i Montplet \cite[equation~(3.4)]{BF14}, Zhang \cite[Proposition~4.2 and Remark~4.3]{Zha24}, and Berman--Boucksom--Guedj--Zeriahi \cite[Lemmas~7.7 and~7.8]{BBGZ13}. The log-determinant/energy correspondence itself belongs to Berman--Boucksom's volume-of-balls theory \cite[Theorem~A]{BB10}. 

Qualitative Bergman approximation in the finite-energy $d_p$ topology
was established by Darvas--Lu--Rubinstein
\cite[Theorem~1.2(i)]{DLR20}. Here we track the errors in the reference-tangent arguments to obtain
a uniform $O((1+B)/k)$ approximation on entropy sublevels by Fubini--Study potentials of degree $m=k+\ell$.
Our approximation uses a fixed auxiliary twist, so its operator
differs from the one in \cite{DLR20}.


We next replace the full Bergman sum by the squared norms of $n+1$ 
independent Gaussian linear combinations of an orthonormal basis,
a construction related to Shiffman--Zelditch's study of random sections
\cite{SZ99}.
This reduces the number of parameters from order $N_k^2$ to order $N_k$,
which is crucial for the subsequent information bound.
Two absolute Gaussian moment calculations, one for each endpoint
Monge--Amp\`ere measure, show that the replacement costs only
$O(k^{-1})$ in expected $d_1$ distance.
Thus the $O((1+B)/k)$ approximation rate survives this reduction.

The covering argument is not a consequence of the section-space dimension $N_k\asymp k^n$ alone. 
Its input is a coarse center $v$ and a potential $u$ close to $v$ in $d_1$. 
In the third part we require sampling points that simultaneously 
control two costs: the difference of  potentials and the 
logarithmic determinant arising in Gaussian interpolation. 
At this stage, inspired by the works 
\cite{Don09}, \cite{BBEGZ19}, and \cite{Ber14,Berm18}, 
we use a fixed-measure balanced Hilbert form to construct a determinantal sample of points which has marginal distribution exactly $\MA(u)$; this uses the projection-process framework of \cite{HKPV06}. 
At each sample point we store an integer rounding of $m(\beta_ku-v)$, where $\beta_ku$ is the Bergman approximation defined in Section~\ref{sec-approximation}. The endpoint estimate of Darvas \cite[Theorem~5.5]{Dar15} controls the expected sum of the absolute integer values. Interpolation from the points and these integers defines a Gaussian mixture depending on $v$ but not on $u$. A relative-entropy comparison with this reference distribution gives a local covering bound. 

For Theorem~\ref{thm-toric-main}, the classical Legendre correspondence \cite[Theorem~2.8]{Abr00} and toric geodesic formula \cite[Section~4, Theorem~3]{Guan99}, together with Darvas's minimizing-geodesic theorem \cite[Theorem~3.5]{Dar15}, give an exact $d_1$--$L^1$ isometry. Relative entropy controls a Sobolev norm on the polytope via the reference transport map. We need to verify the boundary-closure step before applying the convex-function entropy estimate \cite[Theorem~1.1(ii)]{GW17}. The boundary example in Section~\ref{sec-toric} explains why finite entropy cannot simply be replaced by boundedness on the closed polytope.


After the preliminaries, Sections~\ref{sec-l1}--\ref{sec-lower}
prove Theorem~\ref{thm-l1-main} and the lower bound in
Theorem~\ref{thm-d1-main}.
Section~\ref{sec-approximation} establishes the uniform Bergman
approximation estimates.
Sections~\ref{sec-Gaussian}, \ref{sec-sampling},
and~\ref{sec-coding} construct and count local covers.
Section~\ref{sec-toric} treats toric symmetry.
The final section discusses coding interpretations, possible
statistical and numerical applications, and finite-dictionary
approximation of plurisubharmonic potentials.

\medskip
\noindent{\bfseries Acknowledgments.}
The author is supported by JSPS Grant-in-Aid for Scientific Research $(\mathrm{C})$, No.\ 26K06811. 
ChatGPT-6 Astra was used extensively to explore proof candidates and assist with checking calculations, references, and exposition. 
The author directed the research, verified all mathematical arguments, supplied the necessary details, and wrote the final exposition. 
The author takes sole responsibility for the contents of this paper.

\section{Pluripotential and metric preliminaries}
\label{sec-preliminaries}

\subsection{\texorpdfstring{$\omega$}{omega}-plurisubharmonic functions of finite energy}
\label{sec-energy}

In this subsection, we briefly review the necessary facts about the 
finite-energy class of $\omega$-plurisubharmonic functions. 
For a more systematic treatment, see the textbook \cite{GZ17} and the survey \cite{Dar19}. 
An upper semicontinuous function $u:X\to\R\cup\{-\infty\}$ is $\omega$-plurisubharmonic if it is locally the sum of a plurisubharmonic function and a smooth function and if the current $\omega+\ddc u$ is positive.  The set of such functions is denoted by $\PSH(X,\omega)$.

For $u\in\PSH(X,\omega)$, let $u_j:=\max\{u,-j\}$.  The non-pluripolar Monge--Amp\`ere measure is the increasing weak limit
\begin{equation}
\MA(u):=V^{-1}\lim_{j\to\infty}
\mathds{1}_{\{u>-j\}}(\omega+\ddc u_j)^n.
\label{eq-nonpluripolar-ma}
\end{equation}
This construction puts no mass on pluripolar sets and is independent of the truncation procedure; see \cite[Definition~1.1 and Theorem~1.16]{BEGZ10}.  We write $\cE(X,\omega)$ for the full-mass class, characterized by $\MA(u)(X)=1$. 

For bounded $u\in\PSH(X,\omega)$, the Monge--Amp\`ere energy is
\begin{equation}
E(u):=\frac{1}{(n+1)V}\sum_{j=0}^n
\int_Xu\,\omega_u^j\wedge\omega^{n-j}.
\label{eq-e-definition}
\end{equation}
For an arbitrary $\omega$-plurisubharmonic function it is extended by decreasing truncation.  The finite-energy class is
\begin{equation*}
\cE^1(X,\omega):=\{u\in\cE(X,\omega):E(u)>-\infty\}.
\end{equation*}
The class $\cE^1$ is the metric completion of the space of smooth K\"ahler potentials for the path-length metric associated with the normalized Mabuchi Finsler norm
\begin{equation*}
\norm{\xi}_{1,u}:=V^{-1}\int_X\abs{\xi}\,\omega_u^n;
\end{equation*}
see \cite[Theorems~3.5 and~4.17]{Dar15}.  This metric is denoted by $d_1$.

The finite-energy entropy sublevel used below is
\begin{equation}
\cK_B:=\{u\in\cE^1(X,\omega):\sup_Xu=0,\
\Ent(\MA(u)\mid\nu)\leq B\}.
\label{eq-finite-energy-class}
\end{equation}
From this point onward, $\MA$ has the non-pluripolar meaning in \eqref{eq-nonpluripolar-ma}, and we also write $\mu_u=\MA(u)$ for its probability measure when $u\in\cE^1$.

Connectedness and the assumption $n\geq1$ are necessary for the stated normalization.  If $X=X_1\sqcup X_2$ and only $\sup_Xu=0$ is imposed, the functions equal to $0$ on $X_1$ and $-T$ on $X_2$ form an unbounded ray while leaving the Monge--Amp\`ere measure unchanged on each component.  If $n=0$, every normalized class is a singleton.  On a disconnected manifold the conclusions are restored by normalizing separately on every component.

\begin{thm}[{\cite[Lemma~3.45 and Theorems~3.32, 3.46]{Dar19}}]
\label{thm-darvas-inputs}
For every $u\in\PSH(X,\omega)$,
\begin{equation}
\int_Xu\,d\nu\leq\sup_Xu\leq\int_Xu\,d\nu+C_X.
\label{eq-mean-sup}
\end{equation}
If $u,v\in\cE^1(X,\omega)$ and
\begin{equation*}
I_1(u,v):=\int_X\abs{u-v}\,[\MA(u)+\MA(v)],
\end{equation*}
then
\begin{equation}
d_1(u,v)\leq I_1(u,v)\leq 2^{2n+6}d_1(u,v).
\label{eq-d1-i1}
\end{equation}
Moreover, $u_k\to u$ in $d_1$ if and only if $u_k\to u$ in the background $L^1$ topology and $E(u_k)\to E(u)$.  Such convergence implies weak convergence $\MA(u_k)\to\MA(u)$ and
\begin{equation*}
\int_X\abs{u_k-u}\,\MA(v)\longrightarrow0
\end{equation*}
for every $v\in\cE^1(X,\omega)$.
\end{thm}

The powers of $2$ in \eqref{eq-d1-i1} are the constants in \cite[Theorem~3.32]{Dar19} after dividing all Monge--Amp\`ere measures by $V$.  We shall also use the immediate consequences
\begin{equation}
d_1(u,u+c)=\abs c,\qquad
\abs{E(u)-E(v)}\leq d_1(u,v).
\label{eq-d1-constants-e}
\end{equation}
Indeed, the path of constants gives the first upper bound, while differentiating \eqref{eq-e-definition} along a smooth path gives
\begin{equation*}
\frac{d}{dt}E(u_t)=\int_X\dot u_t\,\MA(u_t).
\end{equation*}
The path-length definition of $d_1$ yields the second inequality in \eqref{eq-d1-constants-e}; since $E(u+c)=E(u)+c$, it also yields the reverse bound in the first identity.  Approximation extends the argument to $\cE^1$. 

\begin{thm}[{\cite[Proposition~2.6]{BBEGZ19}}]
\label{thm-bbegz-homeomorphism}
After the normalization $E(u)=0$, the non-pluripolar Monge--Amp\`ere operator is a homeomorphism from $\cE^1(X,\omega)$ with its strong topology to the space of finite-energy probability measures with its strong topology.
\end{thm}
In particular, if $\mu$ is a finite-energy probability measure, there is a unique $u_\mu\in\cE^1(X,\omega)$ satisfying
\begin{equation}
E(u_\mu)=0,\qquad \MA(u_\mu)=\mu.
\label{eq-inverse-ma}
\end{equation}
We denote this normalized solution map by $\mu\mapsto u_\mu$. 

We also recall 
regularization and uniform integrability 
results for $\omega$-psh functions. 

\begin{thm}[{\cite[Theorem~1]{BK07}; \cite[Theorem~A.4]{Dar19}}]
\label{thm-smooth-decreasing}
Every $u\in\PSH(X,\omega)$ is the pointwise decreasing limit of smooth functions $v_j$ satisfying $\omega+\ddc v_j>0$.
\end{thm}

The decreasing smooth $\omega$-plurisubharmonic regularization is due to B\l ocki and Ko\l odziej \cite[Theorem~1]{BK07}; multiplying the approximants by factors tending to one and taking a diagonal subsequence gives strict positivity as recorded in \cite[Theorem~A.4]{Dar19}.  The approximants are not asserted to have zero supremum, so we normalize them explicitly in the proof of Proposition~\ref{prop-l1-entropy}.

\begin{thm}[{\cite[Proposition~1.4]{BBEGZ19}}]
There exist $\alpha_X,C_X>0$ such that
\begin{equation}
\sup_{\substack{\varphi\in\PSH(X,\omega)\\ \sup_X\varphi=0}}
\int_Xe^{-\alpha_X\varphi}\,d\nu\leq C_X.
\label{eq-alpha-integrability}
\end{equation}
\end{thm}

Estimate \eqref{eq-alpha-integrability} is a compact K\"ahler form of Skoda-type exponential integrability; we use the formulation in \cite[Proposition~1.4]{BBEGZ19}.




Finally, \cite[Chapter~I, Proposition~5.9]{Dem} states that $\PSH(\Omega)\cap L^1_{\mathrm{loc}}(\Omega)$ is closed in $L^1_{\mathrm{loc}}$ and that each $L^1_{\mathrm{loc}}$-bounded subset is relatively compact.  Applied in finitely many charts to $u+\rho$, where $\ddc\rho=\omega$, and combined with \eqref{eq-mean-sup}, this gives $L^1$ relative compactness of $\PSH_0(X,\omega)$.  The local submean inequality also gives
\begin{equation}
u_j\longrightarrow u\text{ in }L^1,\quad u_j,u\in\PSH(X,\omega)
\quad\Longrightarrow\quad
\sup_Xu_j\longrightarrow\sup_Xu.
\label{eq-sup-continuity}
\end{equation}
To justify \eqref{eq-sup-continuity}, the Hartogs upper estimate gives the upper limit inequality.  If the lower limit failed by $\delta>0$, the $L^1$ limit would be bounded above almost everywhere by $\sup_Xu-\delta$ along a subsequence; the submean inequality would extend this bound everywhere, a contradiction.

\subsection{Entropy sublevels and the finite-energy closure}
\label{sec-closure}

We first record two elementary facts about relative entropy.  For a probability measure $\mu$ on $X$,
\begin{equation}
\Ent(\mu\mid\nu)=\sup_{g\in C(X)}
\left\{\int_Xg\,d\mu-\log\int_Xe^g\,d\nu\right\}.
\label{eq-entropy-variational}
\end{equation}
For bounded Borel $g$, set $Z_g=\int e^g d\nu$ and $d\nu_g=Z_g^{-1}e^g d\nu$. If $D=\Ent(\mu\mid\nu)<\infty$, nonnegativity of relative entropy gives
\begin{equation*}
0\leq\Ent(\mu\mid\nu_g)=D-\int g\,d\mu+\log Z_g.
\end{equation*}
For $\mu=f\nu$, take $g_M=\max\{-M,\min\{\log f,M\}\}$, with $\log0=-\infty$. Then $e^{g_M}\to f$ and $e^{g_M}\leq1+f$, while $\int g_Md\mu\to\int f\log f d\nu$, possibly $+\infty$: the negative part is dominated by $-f\log f\,\mathds{1}_{\{f<1\}}\in L^1(\nu)$ and the positive part increases. This proves the variational formula over bounded Borel functions, including infinite entropy. If $\mu$ is not absolutely continuous, take $A$ with $\nu(A)=0<\mu(A)$; the tests $M\mathds{1}_A$ make that supremum infinite.

To recover the continuous tests in \eqref{eq-entropy-variational}, approximate each bounded Borel $g$, $\abs{g}\leq M$, by continuous $g_j$, $\abs{g_j}\leq M$, in $L^1(\mu+\nu)$. Inner and outer regularity, continuous approximation of indicators, and truncation give this approximation. Since $\abs{e^a-e^b}\leq e^M\abs{a-b
}$ on $[-M,M]$, both integrals converge. Thus the two suprema agree, and \eqref{eq-entropy-variational} proves weak lower semicontinuity.

\begin{lem}
\label{lem-measurable-entropy}
For probability measures $\mu,\nu$ with $D=\Ent(\mu\mid\nu)<\infty$ and any nonnegative Borel function $g$, allowing the value $+\infty$,
\begin{equation}
 \int e^g d\nu<\infty\quad\Longrightarrow\quad
 \int g\,d\mu\leq D+\log\int e^g d\nu<\infty.
 \label{eq-measurable-entropy}
\end{equation}
\end{lem}
\begin{proof}
Apply the bounded Borel inequality just proved to $g_M=\min\{g,M\}$:
\begin{equation*}
 \int g_Md\mu\leq D+\log\int e^{g_M}d\nu
 \leq D+\log\int e^g d\nu.
\end{equation*}
Monotone convergence proves \eqref{eq-measurable-entropy}. Integrability against $\mu$ is a conclusion, not an assumption. The proof uses only measurability, so the lemma also applies on the polytope in Section~\ref{sec-toric}.
\end{proof}

\begin{thm}[{\cite[Theorem~2.17 and Corollary~2.19]{BBEGZ19}}]
\label{thm-bbegz-entropy}
Every probability measure of finite relative entropy with respect to the smooth measure $\nu$ has finite energy.  For each $B<\infty$, the family
\begin{equation*}
\{\mu:\Ent(\mu\mid\nu)\leq B\}
\end{equation*}
is strongly compact, and weak convergence $\mu_j\rightharpoonup\mu$ inside this family implies
\begin{equation*}
d_1(u_{\mu_j},u_\mu)\longrightarrow0.
\end{equation*}
\end{thm}

The hypothesis in \cite[Theorem~2.17]{BBEGZ19} is that the reference measure is tame, meaning that its pullback to a resolution has an $L^p$ density for some $p>1$.  The smooth measure $\nu$ is tame. 

We shall also need a regularization that does not increase entropy.

\begin{lem}
\label{lem-heat-smoothing}
If $f\geq0$ and $\int_Xf\,d\nu=1$, there are smooth strictly positive probability densities $f_t$ such that $f_t\to f$ in $L^1(\nu)$ as $t\downarrow0$ and
\begin{equation}
\int_Xf_t\log f_t\,d\nu\leq\int_Xf\log f\,d\nu.
\label{eq-heat-entropy}
\end{equation}
\end{lem}

\begin{proof}
Let $P_t$ be the heat semigroup of the Riemannian metric underlying $\omega$, whose volume measure is a constant multiple of $\nu$, and put $f_t=P_tf$.  The heat-kernel construction in \cite[Theorems~7.13 and~7.20]{Gri09} gives a smooth symmetric kernel $p_t(x,y)>0$ for $t>0$ with
\begin{equation*}
P_tf(x)=\int_Xp_t(x,y)f(y)\,d\nu(y),
\qquad \int_Xp_t(x,y)\,d\nu(y)=1.
\end{equation*}
Thus $P_t$ is a positivity-preserving Markov contraction, it preserves $\nu$, and $f_t$ is smooth and strictly positive.  The heat kernels form an approximate identity, so the $L^1$ contraction property and density of continuous functions give
\begin{equation*}
\lim_{t\downarrow0}\norm{P_tf-f}_{L^1(\nu)}=0.
\end{equation*}
Since $s\mapsto s\log s$ is convex, Jensen's inequality for the heat kernel gives
\begin{equation*}
f_t\log f_t\leq P_t(f\log f).
\end{equation*}
Integration against the invariant measure $\nu$ proves \eqref{eq-heat-entropy}.
\end{proof}

Smooth approximation in $d_1$ with convergence of the relative entropy
was established by Berman--Darvas--Lu
\cite[Lemma~3.1 and Theorem~3.2]{BDL17}.
For the closure statement below, we need the approximants to remain
in the same entropy sublevel. The heat regularization of the
Monge--Amp\`ere density in Lemma~\ref{lem-heat-smoothing}
provides this additional property.

\begin{prop}
\label{thm-closure-main}
For every $B\geq0$, the class $\cK_B$ defined in \eqref{eq-finite-energy-class} is $d_1$-compact and
\begin{equation}
\overline{\cK_B^{\cH}}^{\,d_1}=\cK_B.
\label{eq-d1-closure-main}
\end{equation}
For $B=0$, both entropy sublevels consist only of the zero potential.
\end{prop}

\begin{proof}
Theorem~\ref{thm-bbegz-entropy}, applied to the tame measure $\nu$, says that the measure entropy sublevel is strongly compact and that the normalized Monge--Amp\`ere solution map in \eqref{eq-inverse-ma} is continuous there. 
The solutions normalized by $E=0$ form a compact set. 
By Theorem~\ref{thm-darvas-inputs} and \eqref{eq-sup-continuity}, 
the map $u\mapsto u-\sup_Xu$ is continuous in $d_1$. 
Its image is precisely $\cK_B$, which is therefore compact. 

Let $u\in\cK_B$ and write $\MA(u)=f\nu$.  Apply Lemma~\ref{lem-heat-smoothing} and let $u_t\in\cH_\omega$, normalized by $\sup_Xu_t=0$, solve
\begin{equation*}
V^{-1}\omega_{u_t}^n=f_t\nu.
\end{equation*}
Existence and uniqueness follow from 
\cite{Yau78}.  Equation \eqref{eq-heat-entropy} gives $u_t\in\cK_B^{\cH}$.  Since $f_t\nu\to f\nu$ weakly inside the fixed entropy sublevel as $t\downarrow0$, Theorem~\ref{thm-bbegz-entropy} first gives $d_1$ convergence of the solutions normalized by $E=0$.  Their background $L^1$ convergence and \eqref{eq-sup-continuity} show that the constants required to impose supremum zero also converge. Translation invariance of $d_1$ therefore yields $d_1(u_t,u)\to0$.  This proves that the right-hand side of \eqref{eq-d1-closure-main} is contained in the closure.  The converse follows from the weak lower semicontinuity established using \eqref{eq-entropy-variational} and the weak Monge--Amp\`ere convergence in Theorem~\ref{thm-darvas-inputs}.

Finally, $s\log s-s+1\geq0$, with equality only at $s=1$.  Since a probability density has mean one, entropy zero forces $f=1$.  Uniqueness in Theorem~\ref{thm-bbegz-homeomorphism}, followed by the normalization $\sup_Xu=0$, then gives $u=0$.
\end{proof}


\section{Sharp entropy for the background \texorpdfstring{$L^1$}{L1} metric}
\label{sec-l1}

We first isolate the elliptic approximation responsible for the upper bound.  Let $M$ be a fixed compact connected Riemannian manifold of real dimension $d\geq2$, let $dV_M$ be its Riemannian volume measure, and let $\cG$ be the mean-zero Green inverse of its positive Laplacian.  Thus, if $\sigma$ is a finite signed measure of total mass zero,
\begin{equation*}
(\cG\sigma)(x):=\int_MG(x,y)\,d\sigma(y).
\end{equation*}
We use the total variation norm
\begin{equation*}
\norm{\sigma}_{\TV}:=
\sup\left\{\sum_{j=1}^m\abs{\sigma(A_j)}:
M=\bigsqcup_{j=1}^m A_j\text{ is a finite Borel partition}\right\}.
\end{equation*}

\begin{lem}
\label{lem-green-cell}
Let $Q$ be a coordinate cube of diameter $h\leq h_M$, let $c_Q$ be its coordinate center, and let $\sigma_Q$ be a finite signed measure supported in $Q$.  If
\begin{equation}
\int_Qy^\alpha\,d\sigma_Q(y)=0\qquad\text{for }\abs\alpha\leq2,
\label{eq-moment-cancellation}
\end{equation}
then
\begin{equation}
\norm{\int_QG(\,\cdot\,,y)\,d\sigma_Q(y)}_{L^1(M,dV_M)}
\leq C_Mh^2\norm{\sigma_Q}_{\TV}.
\label{eq-green-cell-estimate}
\end{equation}
\end{lem}

\begin{proof}
In local coordinates, a properly supported pseudodifferential parametrix for the Laplacian has principal symbol $(g^{ij}(y)\xi_i\xi_j)^{-1}$.  The recursive symbol construction, followed by elliptic regularity for the smoothing remainder, gives
\begin{equation}
\abs{D_y^\alpha G(x,y)}\leq C_\alpha d(x,y)^{2-d-\abs\alpha}
\label{eq-green-derivatives}
\end{equation}
when $d>2$.  In dimension two, the zero-order estimate is $\abs{G(x,y)}\leq C_M(1+\abs{\log d(x,y)})$, whereas \eqref{eq-green-derivatives} remains valid for $\abs\alpha\geq1$.  These estimates follow by splitting the oscillatory integral at frequency $d(x,y)^{-1}$; the logarithm is precisely the borderline integral of $r^{-1}$.

The case $\alpha=0$ in \eqref{eq-moment-cancellation} says $\sigma_Q(Q)=0$.  Hence, on the near region
\begin{equation*}
U_{C_0h}(Q):=\{x\in M:d(x,Q)<C_0h\},
\end{equation*}
we may write
\begin{equation*}
\int_QG(x,y)\,d\sigma_Q(y)
=\int_Q\bigl(G(x,y)-G(x,c_Q)\bigr)\,d\sigma_Q(y).
\end{equation*}
If $d>2$, the Green bound and polar integration around $y$ and $c_Q$ give, uniformly for $y\in Q$,
\begin{align*}
\int_{U_{C_0h}(Q)}
\bigl(\abs{G(x,y)}+\abs{G(x,c_Q)}\bigr)\,dV_M(x)
&\leq C_M\int_0^{C_1h}r^{2-d}r^{d-1}\,dr\\
&\leq C_Mh^2.
\end{align*}
Integrating first in $x$ and then against $d\abs{\sigma_Q}(y)$ proves the near estimate.  When $d=2$, write locally
\begin{equation*}
G(x,y)=-\gamma_M\log d(x,y)+H(x,y),
\end{equation*}
where $H$ is bounded on the relevant coordinate neighborhood.  With $x=c_Q+h\xi$ and $y=c_Q+h\eta$, the singular part becomes
\begin{equation*}
-\gamma_M\log h-\gamma_M\log\frac{d(c_Q+h\xi,c_Q+h\eta)}{h}.
\end{equation*}
The term $-\gamma_M\log h$ occurs equally in $G(x,y)$ and $G(x,c_Q)$ and disappears because $\sigma_Q(Q)=0$.  The remaining rescaled logarithms are uniformly integrable in $\xi$, while $dV_M(x)$ is bounded by $C_Mh^2d\xi$.  This again gives $C_Mh^2\norm{\sigma_Q}_{\TV}$ on the near region.

We now treat the complement of $U_{C_0h}(Q)$.  Fix $x$ there and let $T_{2,x}$ be the Taylor polynomial of $y\mapsto G(x,y)$ at $c_Q$ through degree two.  Since the coordinate segment
\begin{equation*}
y_t:=c_Q+t(y-c_Q),\qquad 0\leq t\leq1,
\end{equation*}
stays in $Q$, Taylor's theorem with integral remainder gives
\begin{equation}
G(x,y)-T_{2,x}(y)
=\sum_{\abs\alpha=3}\frac{3}{\alpha!}(y-c_Q)^\alpha
\int_0^1(1-t)^2D_y^\alpha G(x,y_t)\,dt.
\label{eq-green-taylor-remainder}
\end{equation}
Put $r:=d(x,Q)$.  Then $d(x,y_t)\geq r$ for all $t$, and \eqref{eq-green-derivatives} with $\abs\alpha=3$ yields
\begin{equation*}
\abs{G(x,y)-T_{2,x}(y)}\leq C_Mh^3r^{-d-1}.
\end{equation*}
Every monomial of $T_{2,x}$ has degree at most two, so its integral against $\sigma_Q$ vanishes by \eqref{eq-moment-cancellation}.  Therefore
\begin{equation*}
\abs{\int_QG(x,y)\,d\sigma_Q(y)}
\leq C_Mh^3r^{-d-1}\norm{\sigma_Q}_{\TV}.
\end{equation*}
Choose a fixed $r_0$ below the coordinate radius.  The volume of a distance shell of radius $r$ is at most $C_Mr^{d-1}dr$, and hence
\begin{align*}
\int_{U_{r_0}(Q)\setminus U_{C_0h}(Q)}
\abs{\int_QG(x,y)\,d\sigma_Q(y)}\,dV_M(x)
&\leq C_Mh^3\norm{\sigma_Q}_{\TV}
\int_{C_0h}^{r_0}r^{-d-1}r^{d-1}\,dr\\
&\leq C_Mh^2\norm{\sigma_Q}_{\TV}.
\end{align*}
On $M\setminus U_{r_0}(Q)$ the third derivatives of $G$ are uniformly bounded, so \eqref{eq-green-taylor-remainder} contributes at most $C_Mh^3\norm{\sigma_Q}_{\TV}$ there.  Adding the near and far estimates proves \eqref{eq-green-cell-estimate}.
\end{proof}

The covering argument below adapts the multiscale approximation
and counting methods of Birman--Solomyak \cite{BirSol67} and
DeVore \cite{Dev98} to a bound on the total variation of the
distributional Laplacian.
The moment-preserving approximation makes this adaptation
explicit in the present setting. 

\begin{lem}
\label{lem-elliptic-cover}
For fixed $A,D>0$, let $\mathscr{F}_{A,D}$ be the set of all $f\in L^1(M,dV_M)$ such that $\Delta f$ is a finite measure and
\begin{equation*}
\norm{\Delta f}_{\TV}\leq A
\quad\text{and}\quad
\abs{\int_Mf\,dV_M}\leq D.
\end{equation*}
There is $C_{M,A,D}>0$ such that
\begin{equation}
H_\e(\mathscr{F}_{A,D},L^1(M,dV_M))
\leq C_{M,A,D}\e^{-d/2}
\label{eq-elliptic-cover}
\end{equation}
for all sufficiently small $\e$.
\end{lem}

\begin{proof}
Choose finitely many coordinate cubes whose interiors cover $M$ and whose closures lie in their charts. Assign every point to the first cube containing it, in a fixed ordering, to obtain a Borel partition $\mathscr{Q}_0$. Subdivide each original coordinate cube dyadically, using half-open faces with a fixed boundary convention, and intersect these cubes with its assigned Borel set. This gives nested Borel partitions $\mathscr{Q}_j$. Each nonempty cell $Q$ is paired with its containing closed coordinate cube $\widetilde Q$, of diameter at most $C_M2^{-j}$. Each cell has at most $2^d$ children, and their containing cubes lie in $\widetilde Q$. The cells need not themselves be cubes: interpolation nodes below lie in $\widetilde Q$, and the Green estimate will be applied on this containing cube.

Let $\delta_x$ denote the Dirac probability measure at $x$. Identify $\widetilde Q$ affinely with $[0,1]^d$, let $L_0,L_1,L_2$ be the one-variable Lagrange polynomials at $0,1/2,1$, and let $z_{Q,b}$ be the node with normalized coordinates $b/2$ for $b\in\{0,1,2\}^d$. For $x\in Q$, first define the local replacement on Dirac masses by
\begin{equation}
R_Q\delta_x:=\sum_{b\in\{0,1,2\}^d}
\left(\prod_{r=1}^dL_{b_r}(x_r)\right)\delta_{z_{Q,b}},
\label{eq-lagrange-replacement}
\end{equation}
where $x_r$ are the normalized coordinates of $x$.  For a finite signed measure $\tau$ supported in $Q$, extend \eqref{eq-lagrange-replacement} by
\begin{equation*}
R_Q\tau:=\int_QR_Q\delta_x\,d\tau(x).
\end{equation*}
Explicitly, $L_0(t)=2(t-1/2)(t-1)$, $L_1(t)=4t(1-t)$, and $L_2(t)=2t(t-1/2)$, so $L_a(b/2)=\delta_{ab}$. 
Applying the one-variable interpolation identity successively in each coordinate gives 
\begin{equation*}
 p(x)=\sum_{b\in\{0,1,2\}^d}p(z_{Q,b})\prod_{r=1}^dL_{b_r}(x_r) 
\end{equation*}
for every coordinate polynomial $p$ whose degree in each variable is at most two.  
Integrating this finite sum against $\tau$, and using the defining property of Dirac measures in \eqref{eq-lagrange-replacement}, gives
\begin{equation*}
\int_{\widetilde Q}p\,d(R_Q\tau)=\int_Qp\,d\tau. 
\end{equation*}
Moreover, boundedness of the Lagrange basis on $[0,1]^d$ gives
\begin{equation*}
\norm{R_Q\tau}_{\TV}\leq C_d\norm{\tau}_{\TV}.
\end{equation*}

Let $\cM(M)$ be the space of finite signed Borel measures and define
\begin{equation*}
P_j:\cM(M)\longrightarrow\cM(M),\qquad
P_j\sigma:=\sum_{Q\in\mathscr{Q}_j}R_Q(\sigma\vert_Q).
\end{equation*}
Thus the image $\operatorname{Im}P_j$ consists of atomic measures supported on the finitely many interpolation nodes of level-$j$ cells.  In particular,
\begin{equation}
\dim\operatorname{Im}P_j\leq C2^{dj},\qquad
\norm{P_j\sigma}_{\TV}\leq C\norm{\sigma}_{\TV}.
\label{eq-pj-rank}
\end{equation}

We next spell out the cellwise cancellation.  If $\operatorname{ch}(Q)$ is the set of children of $Q\in\mathscr{Q}_j$, put
\begin{equation*}
\rho_{j,Q}:=
\sum_{Q'\in\operatorname{ch}(Q)}R_{Q'}(\sigma\vert_{Q'})
-R_Q(\sigma\vert_Q).
\end{equation*}
Both terms in this difference have the same moments as $\sigma\vert_Q$ through degree two, since child and parent coordinates differ affinely. Hence $\rho_{j,Q}$ is supported in $\widetilde Q$ and satisfies \eqref{eq-moment-cancellation} on that cube. Also,
\begin{equation}
(P_{j+1}-P_j)\sigma=\sum_{Q\in\mathscr{Q}_j}\rho_{j,Q},
\qquad
\sum_{Q\in\mathscr{Q}_j}\norm{\rho_{j,Q}}_{\TV}
\leq C_d\norm{\sigma}_{\TV}.
\label{eq-block-cell-decomposition}
\end{equation}
At the terminal level $k$ one similarly has
\begin{equation*}
(I-P_k)\sigma=
\sum_{Q\in\mathscr{Q}_k}(I-R_Q)(\sigma\vert_Q),
\end{equation*}
and the summands have the same moment cancellation on $\widetilde Q$ and the same total variation bound. Applying Lemma~\ref{lem-green-cell} on each containing cube to the summands in \eqref{eq-block-cell-decomposition}, and then using the triangle inequality, gives
\begin{equation}
\norm{\cG(P_{j+1}-P_j)\sigma}_{L^1}\leq C2^{-2j}\norm{\sigma}_{\TV},
\qquad
\norm{\cG(I-P_k)\sigma}_{L^1}\leq C2^{-2k}\norm{\sigma}_{\TV}.
\label{eq-green-blocks}
\end{equation}

Now write $\sigma=\Delta f$ and define
\begin{equation*}
\overline f:=\frac{1}{\vol(M)}\int_Mf\,dV_M.
\end{equation*}
Since $\sigma(M)=0$ and $\cG$ is the inverse of $\Delta$ on mean-zero distributions, $f-\overline f=\cG\sigma$.  The telescoping identity for the $P_j$ therefore gives
\begin{equation}
f=\overline f+\cG P_0\sigma+
\sum_{j=0}^{k-1}\cG(P_{j+1}-P_j)\sigma+
\cG(I-P_k)\sigma.
\label{eq-green-decomposition}
\end{equation}
Choose the least nonnegative integer $k$ for which $C_MA2^{-2k}\leq\e/4$, with $C_M$ valid in \eqref{eq-green-blocks}. This choice is uniform in $f$ and bounds the final term of \eqref{eq-green-decomposition} by $\e/4$. Minimality gives, for sufficiently small $\e$,
\begin{equation*}
C_MA2^{-2k}\leq\e/4,
\qquad 2^{dk}\leq C_{M,A}\e^{-d/2}.
\end{equation*}
Approximate the level-$j$ block with accuracy
\begin{equation*}
\delta_j:=\frac{\e}{8}2^{-(k-j)}.
\end{equation*}
The sum of these accuracies is at most $\e/4$.

An $m$-dimensional normed ball of radius $R$ has a $\delta$-net of cardinality at most $(1+2R/\delta)^m$, by comparing the disjoint $\delta/2$-balls around a maximal separated set with the ball of radius $R+\delta/2$.  Applying this estimate with \eqref{eq-pj-rank} and \eqref{eq-green-blocks}, and writing $\ell=k-j$, gives
\begin{equation*}
\log N_{\delta_j}\leq C_{M,A}2^{dk}2^{-d\ell}(1+\ell),
\qquad \ell:=k-j.
\end{equation*}
Indeed, the block lies in the image under $\cG$ of $\operatorname{Im}P_j+\operatorname{Im}P_{j+1}$, whose dimension is at most $C_M2^{dj}=C_M2^{dk}2^{-d\ell}$, and its radius divided by $\delta_j$ is at most $C_{M,A}2^{3\ell}$.  The series $\sum_{\ell\geq1}2^{-d\ell}(1+\ell)$ converges.  The mean $\overline f$ ranges over a bounded interval and $\cG P_0\sigma$ ranges over a fixed-dimensional bounded set, so their logarithmic covering costs are $O_{M,A,D}(\log(1/\e))$.  They are absorbed by $C_{M,A,D}\e^{-d/2}$.  Together with the two error allocations of size $\e/4$, this proves \eqref{eq-elliptic-cover}.  The geometrically weighted accuracies $\delta_j$ are precisely what prevent a terminal-scale logarithm.
\end{proof}

The following constant-weight packing is obtained by Gilbert's greedy deletion argument \cite[Theorem~1, p.~507]{Gil52}; compare the bound of Varshamov \cite{Var57}. 

\begin{lem}
\label{lem-constant-weight-code}
There are universal constants $\gamma,c>0$ such that, for every sufficiently large integer $N$, one can find a set $\mathscr{C}_N\subset\{0,1\}^N$ with the following properties: every $\theta\in\mathscr{C}_N$ has exactly $\lfloor N/2\rfloor$ entries equal to one,
\begin{equation*}
d_H(\theta,\theta'):=\#\{i:\theta_i\neq\theta_i'\}\geq\gamma N
\end{equation*}
for distinct $\theta,\theta'\in\mathscr{C}_N$, and $\log\#\mathscr{C}_N\geq cN$.
\end{lem}

\begin{proof}
The central binomial coefficient is the largest of the $N+1$ coefficients whose sum is $2^N$. Thus the layer of words of Hamming weight $\lfloor N/2\rfloor$ contains at least $2^N/(N+1)$ elements. Starting with this layer, choose one word and delete all remaining words at Hamming distance less than $\gamma N$; then repeat. Each chosen word deletes at most
\begin{equation*}
\sum_{j\leq\gamma N}\binom Nj\leq\exp(H(\gamma)N)
\end{equation*}
words, where $H(\gamma)=-\gamma\log\gamma-(1-\gamma)\log(1-\gamma)$ and $0<\gamma\leq1/2$. Indeed, for $j\leq\lfloor\gamma N\rfloor$, monotonicity of $(\gamma/(1-\gamma))^j$ gives $\gamma^j(1-\gamma)^{N-j}\geq e^{-NH(\gamma)}$. The binomial theorem therefore gives
\begin{equation*}
1\geq\sum_{j=0}^{\lfloor\gamma N\rfloor}\binom Nj\gamma^j(1-\gamma)^{N-j}
\geq e^{-NH(\gamma)}\sum_{j=0}^{\lfloor\gamma N\rfloor}\binom Nj.
\end{equation*}
Choose $\gamma>0$ so small that $H(\gamma)<(\log2)/2$. Dividing the initial layer size by the maximal deleted-ball size proves that at least $\exp(cN)$ words are selected for all sufficiently large $N$.
\end{proof}

\begin{prop}
\label{prop-l1-entropy}
The conclusions of Theorem~\ref{thm-l1-main} hold.
\end{prop}

\begin{proof}
For $u\in\PSH(X,\omega)$, define the distributional Laplacian by
\begin{equation*}
(\Delta_\omega u)\omega^n=n\ddc u\wedge\omega^{n-1}.
\end{equation*}
Positivity gives the measure identity
\begin{equation}
(n+\Delta_\omega u)\omega^n
=n\omega_u\wedge\omega^{n-1},\qquad
\int_X(n+\Delta_\omega u)\omega^n=nV.
\label{eq-positive-trace}
\end{equation}
The second equality in \eqref{eq-positive-trace} follows from Stokes' theorem, first for smooth regularizations and then by weak convergence.  Since the measure on the right is positive and has mass $nV$, \eqref{eq-positive-trace} implies
\begin{equation}
\norm{(\Delta_\omega u)\omega^n}_{\TV}\leq2nV.
\label{eq-laplacian-tv}
\end{equation}
If $\sup_Xu=0$, then \eqref{eq-mean-sup} gives
\begin{equation}
-C_X\leq\int_Xu\,d\nu\leq0.
\label{eq-normalized-mean}
\end{equation}
Lemma~\ref{lem-elliptic-cover}, applied with $d=2n$ using \eqref{eq-laplacian-tv} and \eqref{eq-normalized-mean}, proves the upper bound in \eqref{eq-l1-main}.

We give all geometric details of the lower bound.  Choose a holomorphic chart $z:U\to B(0,2)\subset\R^{2n}$ and a nonzero function $\psi\in C_c^\infty(B(0,1))$ with $\psi\leq0$.  On the smaller chart the normalized volume has the form $d\nu=J(z)\,dz$, where compactness gives constants $0<c_0<C_0$ such that
\begin{equation*}
c_0\leq J(z)\leq C_0.
\end{equation*}
For every sufficiently small grid size $h$, choose sites $x_1,\ldots,x_N$ whose coordinate balls $B(z(x_i),h)$ are disjoint and contained in $B(0,1)$.  A maximal grid gives constants $c_1,C_1>0$ with
\begin{equation*}
c_1h^{-2n}\leq N\leq C_1h^{-2n}.
\end{equation*}
For a fixed amplitude $a>0$, define the smooth function
\begin{equation*}
\varphi_{i,h}(x):=
\begin{cases}
ah^2\psi\bigl((z(x)-z(x_i))/h\bigr),&x\in U,\\
0,&x\notin U.
\end{cases}
\end{equation*}
Because the supports are disjoint and $\norm{\ddc\varphi_{i,h}}_{C^0}\leq C_Xa$, choosing $a$ once and for all sufficiently small gives
\begin{equation*}
u_\theta:=\sum_{i=1}^N\theta_i\varphi_{i,h},
\qquad \omega_{u_\theta}\geq\omega/2
\end{equation*}
for every binary word $\theta=(\theta_1,\ldots,\theta_N)$.  Each $u_\theta$ is nonpositive and vanishes outside the bump supports, so $\sup_Xu_\theta=0$ and $u_\theta\in\cH_\omega$.

If $I(\theta,\theta'):=\{i:\theta_i\neq\theta_i'\}$, disjointness of the supports and the lower bound for $J$ give
\begin{align}
d_{L^1}(u_\theta,u_{\theta'})
&=\sum_{i\in I(\theta,\theta')}
\int_X\abs{\varphi_{i,h}}\,d\nu\notag\\
&\geq c_Xah^{2n+2}\#I(\theta,\theta')
=c_Xah^{2n+2}d_H(\theta,\theta').
\label{eq-l1-hamming}
\end{align}
Apply Lemma~\ref{lem-constant-weight-code}.  Its distinct codewords differ at at least $\gamma N$ sites, so \eqref{eq-l1-hamming} and the lower bound for $N$ imply
\begin{equation*}
d_{L^1}(u_\theta,u_{\theta'})\geq c_Xah^2.
\end{equation*}
Now fix $0<\e<\e_X$ and choose an admissible grid scale $h$ satisfying
\begin{equation*}
2\e<c_Xah^2\leq8\e.
\end{equation*}
After decreasing $\e_X$, such a dyadic grid scale exists.  The code potentials form a $2\e$-packing, while
\begin{equation*}
\log\#\mathscr{C}_N\geq cN\geq c_Xh^{-2n}\geq c_X\e^{-n}.
\end{equation*}
Every open $\e$-ball contains at most one packing point, which proves the lower bound in \eqref{eq-l1-main}.

It remains to identify the closure.  Given $u\in\PSH_0(X,\omega)$, choose $v_j$ as in Theorem~\ref{thm-smooth-decreasing} and put $s_j:=\sup_Xv_j$.  Since $v_j\downarrow u$, one has $s_j\downarrow0$: if the limit were positive, a cluster point of maximizers would give a point where $u$ is positive.  Thus $v_j-s_j\in\cH_\omega$ and converges to $u$ in $L^1$.  Conversely, an $L^1$ limit of normalized $\omega$-psh functions is $\omega$-psh by \cite[Chapter~I, Proposition~5.9]{Dem}, and \eqref{eq-sup-continuity} preserves the normalization.  This proves \eqref{eq-l1-closure}. 
The smooth packing lies in both classes, while the largest class has the preceding cover, so the two-sided estimate holds for each class.
\end{proof}

\section{The \texorpdfstring{$d_1$}{d1} lower bound in arbitrary dimension}
\label{sec-lower}

\begin{prop}
\label{prop-d1-lower}
The lower bounds \eqref{eq-d1-lower-main}--\eqref{eq-d1-lower-range} hold for both $\cK_B^{\cH}$ and $\cK_B$.
\end{prop}

\begin{proof}
Use the chart, the nonpositive bump $\psi$, and the disjoint sites from the proof of Proposition~\ref{prop-l1-entropy}.  Fix a geometric constant $a_0>0$ and, for the given entropy budget, set the Hessian amplitude
\begin{equation}
a:=a_0\min\{1,\sqrt B\}.
\label{eq-ab-definition}
\end{equation}
For every sufficiently small $h$, the number $N$ of disjoint balls satisfies $c_Xh^{-2n}\leq N\leq C_Xh^{-2n}$.  For a binary word $\theta=(\theta_1,\ldots,\theta_N)$ set
\begin{equation}
u_\theta(x):=\sum_{j=1}^N\theta_j ah^2
\psi\left(\frac{z(x)-z(x_j)}{h}\right),
\label{eq-d1-bumps}
\end{equation}
where each summand is extended by zero outside its coordinate support.  All potentials in \eqref{eq-d1-bumps} use the same sites and scale.  Choose $a_0$ so small that
\begin{equation}
\frac12\omega\leq\omega_{u_\theta}\leq\frac32\omega
\label{eq-uniform-kahler}
\end{equation}
for every word and every sufficiently small scale.  If $\rho_\theta:=\omega_{u_\theta}^n/\omega^n$, smoothness of the determinant on the compact matrix interval in \eqref{eq-uniform-kahler} gives
\begin{equation}
2^{-n}\leq\rho_\theta\leq(3/2)^n,
\qquad \abs{\rho_\theta-1}\leq C_Xa.
\label{eq-density-bounds}
\end{equation}
The first two bounds follow by comparing the eigenvalues of $\omega_{u_\theta}$ with those of $\omega$; the last follows from the first derivative bound for the determinant on the matrix segment joining the identity to $\omega^{-1}\omega_{u_\theta}$.  The density equals one off the bump supports and has mean one.

We record the Taylor estimate explicitly.  Let $F(s):=s\log s-s+1$.  Then $F(1)=F'(1)=0$ and $F''(s)=1/s$.  For $s\in[2^{-n},(3/2)^n]$, Taylor's formula with integral remainder gives
\begin{align*}
F(s)
&=(s-1)^2\int_0^1(1-t)F''(1+t(s-1))\,dt\\
&\leq2^{n-1}(s-1)^2.
\end{align*}
Since $\int_X(\rho_\theta-1)d\nu=0$, the definition of relative entropy, this Taylor estimate, and \eqref{eq-density-bounds} yield
\begin{align}
\Ent(\mu_{u_\theta}\mid\nu)
&=\int_X(\rho_\theta\log\rho_\theta-\rho_\theta+1)d\nu\notag\\
&\underalign{\eqref{eq-density-bounds}}{\leq} C_Xa^2
\underalign{\eqref{eq-ab-definition}}{\leq} B.
\label{eq-bump-kl}
\end{align}
Thus, \eqref{eq-bump-kl} places every code potential in $\cK_B^{\cH}$.
The last inequality follows after decreasing $a_0$ once: for $B<1$ it uses $a^2=a_0^2B$, while for $B\geq1$ it uses $a^2=a_0^2\leq B$.

Let $I(\theta,\theta'):=\{j:\theta_j\neq\theta_j'\}$ and put $k:=\#I(\theta,\theta')=d_H(\theta,\theta')$.  On the support of the $j$-th differing bump, $\abs{u_\theta-u_{\theta'}}=ah^2\abs{\psi((z-z(x_j))/h)}$.  The lower density bound in \eqref{eq-density-bounds}, the coordinate Jacobian bound, and disjointness of the supports therefore give
\begin{align}
I_1(u_\theta,u_{\theta'})
&=\int_X\abs{u_\theta-u_{\theta'}}
(\mu_{u_\theta}+\mu_{u_{\theta'}})\notag\\
&\geq c_X\sum_{j\in I(\theta,\theta')}
ah^2\int_{B(z(x_j),h)}
\abs{\psi((z-z(x_j))/h)}\,dz\notag\\
&=c_Xakh^{2n+2}.
\label{eq-bump-i1}
\end{align}
The right-hand inequality in \eqref{eq-d1-i1}, namely $I_1\leq2^{2n+6}d_1$, changes \eqref{eq-bump-i1} into
\begin{equation}
d_1(u_\theta,u_{\theta'})
\geq 2^{-2n-6}c_Xakh^{2n+2}.
\label{eq-bump-d1}
\end{equation}
Restrict to the code $\mathscr{C}_N$ from Lemma~\ref{lem-constant-weight-code}.  Then $k\geq\gamma N$, so \eqref{eq-bump-d1} and $N\geq c_Xh^{-2n}$ separate distinct code potentials by $c_Xah^2$, and $\log\#\mathscr{C}_N\geq c_Xh^{-2n}$.

Now fix $\e$ in the range \eqref{eq-d1-lower-range} and choose an admissible dyadic scale $h$ such that the lower separation is greater than $2\e$ and $ah^2\leq C_X\e$.  Such a scale exists because \eqref{eq-d1-lower-range} ensures that $\e/a$ is below the fixed chart scale.  Equations \eqref{eq-ab-definition} and \eqref{eq-bump-d1} give
\begin{equation*}
\log\#\mathscr{C}_N\geq c_X(a/\e)^n
=c_Xa_0^n\min\{1,B^{n/2}\}\e^{-n}.
\end{equation*}
Every $\e$-ball contains at most one code potential, so this proves \eqref{eq-d1-lower-main} for the smooth class.  For a smooth potential, the non-pluripolar measure \eqref{eq-nonpluripolar-ma} equals $V^{-1}\omega_u^n$; hence $\cK_B^{\cH}\subset\cK_B$.  The same separated family lies in the finite-energy class, so its packing gives the same lower bound even when covering centers may be chosen in that larger class.
\end{proof}

\section{Uniform Bergman approximation on entropy sublevels}
\label{sec-approximation}
Assume $\omega=c_1(L,h_0)$ from now through Section~\ref{sec-coding}, where $L$ is an ample holomorphic line bundle and $h_0$ is its fixed smooth positively curved Hermitian metric. In a local holomorphic frame $e$ of $L$, define the real weight $\phi_0$ by $\abs{e}_{h_0}^2=e^{-\phi_0}$. We use $\ddc=\frac{\sqrt{-1}}{2\pi}\partial\bar\partial$, so $\ddc\phi_0=\omega$. Recall that $\mu_u=\MA(u)$.


Darvas--Lu--Rubinstein \cite[Theorem~1.2 (i)]{DLR20} established
Bergman approximation in the finite-energy $d_p$ topology.
The purpose of this section is to prove a uniform quantitative estimate:
for every $u\in\cK_B$, the $d_1$ error of the Bergman approximation
used here is bounded by $C_X(1+B)/k$.
This approximation incorporates a fixed auxiliary twist and is not
identical to the operator in \cite{DLR20}.
Two energy comparisons supply the estimate. We first explain their
antecedents and then prove the quantitative versions needed here. 

\subsection{Energy formulas and entropy control}
For smooth semipositive potentials $u,v$, the energy difference formula is
\begin{equation}
 E(v)-E(u)=\frac1{(n+1)V}\sum_{j=0}^n
 \int_X(v-u)\omega_v^j\wedge\omega_u^{n-j}.
 \label{eq-gu-energy-difference}
\end{equation}
Energy is increasing and concave. In particular,
\begin{equation}
 E(v)-E(u)\leq\int_X(v-u)d\mu_u.
 \label{eq-gu-concavity}
\end{equation}
These identities are \cite[Proposition~2.1]{BBGZ13}; decreasing extension gives \eqref{eq-gu-concavity} in $\cE^1$ when the integral is finite \cite[Proposition~2.4]{BBGZ13}. Write $S=\{\tau\in\mathbb C:0<\mathrm{Re}\,\tau<1\}$ and $\pi_X:X\times S\to X$, $(x,\tau)\mapsto x$. A subgeodesic $u_t$ is a family whose imaginary-time-independent extension $U(x,\tau)=u_{\mathrm{Re}\,\tau}(x)$ satisfies $\pi_X^*\omega+\ddc U\geq0$ on $X\times S$. The fiber-integral formula \cite[Proposition~6.2]{BBGZ13} makes $E(u_t)$ convex for a smooth subgeodesic and affine for a bounded weak geodesic. The latter solves the homogeneous Monge--Amp\`ere equation with its prescribed boundary values. Between smooth strictly positive endpoints its first derivatives are bounded and
\begin{equation*}
 \frac d{dt}E(u_t) \bigg\vert_{0+}=\int_X\dot u_0d\mu_{u_0}.
\end{equation*}
See \cite[Section~2.2]{DLR20}; bounded derivatives permit endpoint variation in \eqref{eq-gu-energy-difference}.

The rooftop formula \cite[Corollary~4.14]{Dar15} implies
\begin{equation}
 u\leq v\quad\Longrightarrow\quad d_1(u,v)=E(v)-E(u).
 \label{eq-gu-ordered}
\end{equation}
Indeed the rooftop envelope of ordered endpoints is the smaller endpoint. For arbitrary $u,v$, apply the triangle inequality through $z=\max(u,v)$, and then \eqref{eq-gu-concavity} at each endpoint:
\begin{align*}
 d_1(u,v)&\leq E(z)-E(u)+E(z)-E(v)\\
 &\leq\int_X(v-u)_+d\mu_u+\int_X(u-v)_+d\mu_v
 \leq I_1(u,v).
\end{align*}
Here $I_1(u,v)=\int_X\abs{u-v}(d\mu_u+d\mu_v)$, as in Theorem~\ref{thm-darvas-inputs}; that theorem also gives $I_1(u,v)\leq C_nd_1(u,v)$.

For $v\in\PSH_0(X,\omega)$, the function $-\alpha_Xv$ is nonnegative Borel, and \eqref{eq-alpha-integrability} bounds its exponential integral. Lemma~\ref{lem-measurable-entropy} therefore applies, even if $v$ is discontinuous or equals $-\infty$. For any probability $\mu$ with $\Ent(\mu\mid\nu)\leq B$, it gives
\begin{equation}
 \int_X(-v)d\mu
 \leq\frac{B+\log C_X}{\alpha_X}\leq C_X(1+B),
 \qquad v\in\PSH_0(X,\omega).
 \label{eq-gu-entropy-control}
\end{equation}
For a nonzero $t\in H^0(X,mL)$, put $v_t=m^{-1}(\log\abs{t}_{h_0^m}^2-\sup_X\log\abs{t}_{h_0^m}^2)$. The Poincar\'e--Lelong formula \cite[equation~(3.11)]{Dem11}, in our $\ddc$ convention, gives $\omega+\ddc v_t=m^{-1}[\operatorname{div}(t)]\geq0$. Hence \eqref{eq-gu-entropy-control} applies to $v_t$, proving logarithmic integrability at the divisor. For $u\in\cK_B$, monotonicity and \eqref{eq-gu-concavity} with $v=0$ give
\begin{equation}
 0\leq-E(u)\leq\int_X(-u)d\mu_u\leq C_X(1+B).
 \label{eq-energy-budget}
\end{equation}

\subsection{Hilbert forms and the quantized energy}
Choose a fixed positive integer $\ell$ such that $\ell\omega+\Ric(\nu)\geq\omega$, where locally $\Ric(\nu)=-\ddc\log\nu$. Keep $m=k+\ell$. Fix an integer $k_0\geq1$, depending only on the reference geometry, such that $mL$ is very ample for every $k\geq k_0$ and the reference expansion below holds on this range. This fixed threshold may be increased independently of $u$ and $B$. Set $N=N_k=h^0(X,mL)$. We use $m$ for exact tensor-power normalizations, but express error bounds and growth estimates in terms of $k$; the ratios $m/k$ and $k/m$ are bounded by geometric constants.

Choose sections $s_1,\ldots,s_N$ orthonormal for $L^2(h_0^m,\nu)$. We explain how evaluation becomes a numerical column. In a local holomorphic frame $e$ of $L$, write $s_j=\sigma_j e^{\otimes m}$ and put $\sigma=(\sigma_j)_{j=1}^N$. In the associated unit frame, the evaluation column is
\begin{equation}
 s(x)=e^{-m\phi_0(x)/2}\sigma(x),\qquad
 \int_Xs(x)s(x)^*d\nu(x)=I_N .
 \label{eq-evaluation-column}
\end{equation}
Thus $\sigma$ is holomorphic on the chart, whereas $s$ records pointwise norms. Unit-frame changes multiply the whole column by a scalar of modulus one. A finite Borel partition subordinate to charts fixes measurable columns wherever a matrix-valued probability law is needed.

For $u\in\cE^1$, define the Hilbert form, its associated Bergman potential, and the quantized energy by
\begin{align}
 H_u=H_k(u)&:=\int_Xss^*e^{-ku}d\nu,
 \label{eq-gu-hilbert}\\
 u_H&:=\frac1m\log\frac{s^*H^{-1}s}{N},
 \qquad E_k(H):=-\frac1{mN}\log\det H,\quad H>0,
 \label{eq-gu-fs}\\
 \beta_ku&:=u_{H_k(u)}.
 \label{eq-bergman-map}
\end{align}
The difference $m-k=\ell$ between the tensor power and the Hilbert exponent in \eqref{eq-gu-hilbert} is essential to the curvature argument below. The subscript of $E_k$ records the approximation index; its exact normalization in \eqref{eq-gu-fs} is $1/(mN)$. 

Full-mass potentials have zero Lelong numbers, so $e^{-pu}\in L^1(\nu)$ for every finite $p>0$ by \cite[proof of Lemma~6.4]{BBGZ13}. Hence $H_k(u)$ is finite. For a nonzero coefficient vector, the corresponding section is nonzero on an open set, so its squared integral is strictly positive. Thus $H_k(u)>0$. Also $u_H$ is a global smooth strictly $\omega$-psh potential: the phase in \eqref{eq-evaluation-column} cancels, and
\begin{equation}
 m\omega_{u_H}=\ddc\log(\sigma^*H^{-1}\sigma).
 \label{eq-fs-local-curvature}
\end{equation}

For the reference $H=I_N$, write $\rho_m=s^*s$. 
The smooth Bergman expansion \cite[Theorem~1]{Zel98} gives
\begin{equation}
 N_k\asymp k^n,\qquad \rho_m/N_k=1+O(k^{-1}),
 \qquad u_I=O(k^{-2}).
 \label{eq-gu-reference-expansion}
\end{equation}
All remainder estimates in these expansions hold in every fixed
$C^q$ norm. 
To see the last rate explicitly, the fixed smooth expansion has the form
\begin{align*}
 \rho_m&=a_0m^n+a_1(x)m^{n-1}+O(m^{n-2}),\\
 N_k=\int_X\rho_m d\nu&=a_0m^n+\overline a_1m^{n-1}+O(m^{n-2}),
 \qquad\overline a_1=\int_Xa_1d\nu.
\end{align*}
Here $a_0>0$ is constant because the reference measure is proportional to the curvature volume. Dividing, taking the logarithm, and then using the factor $1/m$ in \eqref{eq-gu-fs} yields
\begin{equation*}
u_I=\frac{a_1-\overline a_1}{a_0m^2}+O(m^{-3}).
\end{equation*}
Thus $\norm{u_I}_{C^q}\leq C_qk^{-2}$ for every fixed $q$. Without division by $N_k$, a constant term of order $(\log k)/k$ would remain. This expansion concerns only the fixed smooth reference metric, not a family of singular inputs.

\subsection{Energy comparisons}
Normalized logarithmic Hilbert determinants approximate Monge--Amp\`ere energy in Berman--Boucksom's volume-of-balls approach \cite[Theorem~A]{BB10}. The tangent-line mechanism needed here is more specific. Zhang's quantization argument \cite[Proposition~4.2]{Zha24} compares the energy of a potential with that of a Hilbert approximation after a small contraction of the potential. His Remark~4.3 traces the convexity argument to Berman--Freixas i Montplet \cite[equation~(3.4)]{BF14}. Our first lemma uses this mechanism with contraction factor $k/m=1-\ell/m$, retains the reference-density error of order $k^{-1}$, and extends the resulting bound to all finite-energy endpoints. 

\begin{lem}
\label{lem-gu-hilbert}
For every $u\in\cE^1$ with $u\leq0$,
\begin{equation}
 0\leq-E_k(H_k(u))\leq(1+C_X/k)(-E(u)).
 \label{eq-gu-hilbert-bound}
\end{equation}
\end{lem}
\begin{proof}
We use the adjoint Hilbert-form comparison arising from
Berndtsson's direct-image positivity
\cite[Proposition~3.1]{Ber09}, in the finite-energy formulation
of Darvas--Lu--Rubinstein \cite[Proposition~2.12]{DLR20}. 
Let $(F,h_F)$ have curvature $\eta>0$, and let $v_t$ be a finite-energy $\eta$-subgeodesic. Denote by $\cH_t$ the adjoint $L^2$ form on $H^0(X,F\otimes K_X)$, obtained by integrating the squared norm of an $F$-valued top form against $h_Fe^{-v_t}$. The cited result gives
\begin{equation*}
 \cH_t\geq
 \cH_0^{1/2}(\cH_0^{-1/2}\cH_1\cH_0^{-1/2})^t\cH_0^{1/2}.
\end{equation*}
For positive Hermitian forms, $A\geq B$ implies $\det A\geq\det B$, by applying the eigenvalue comparison to $B^{-1/2}AB^{-1/2}$. Taking determinants therefore gives
\begin{equation*}
 \log\det\cH_t\geq(1-t)\log\det\cH_0+t\log\det\cH_1.
\end{equation*}
The same comparison on every subinterval proves concavity of $\log\det\cH_t$.

Take $F=mL\otimes K_X^{-1}$, with reference curvature $m\omega+\Ric(\nu)$; then $F\otimes K_X=mL$ and the adjoint form is exactly $H_k(u_t)$. Let $u_t$ be the weak geodesic from $0$ to a smooth strictly positive endpoint $u\leq0$, and use its strip extension $U$ defined above. The curvature condition is verified on the product, including mixed time-space components:
\begin{equation*}
 \pi_X^*(m\omega+\Ric(\nu))+\ddc(kU)
 =k(\pi_X^*\omega+\ddc U)+\pi_X^*(\ell\omega+\Ric(\nu))\geq0.
\end{equation*}
All weights and velocities are bounded. Consequently $t\mapsto-E_k(H_k(u_t))$ is concave, vanishes at zero, and is bounded by its initial derivative. Pointwise convexity of $u_t$ gives $\dot u_0\leq0$. Differentiating the determinant at $H_k(0)=I_N$ gives
\begin{align*}
 -E_k(H_k(u))
 &\leq\frac{k}{m}\int_X(-\dot u_0)\frac{\rho_m}{N}d\nu\\
 &\underalign{\eqref{eq-gu-reference-expansion}}{\leq}
 \frac{k}{m}(1+C_X/k)\int_X(-\dot u_0)d\nu\\
 &\leq(1+C_X/k)(-E(u)).
\end{align*}
The last line uses affinity of energy on $u_t$ and $k/m\leq1$. Since $e^{-ku}\geq1$, \eqref{eq-evaluation-column} gives $H_k(u)\geq I_N$, proving the other inequality.

\emph{Passage to finite energy.}
Smooth semipositive endpoints follow by replacing $u$ by $(1-\e)u$ and taking the limit. For arbitrary $u\leq0$, Theorem~\ref{thm-smooth-decreasing} gives smooth $u_j\downarrow u$. Set $v_j=u_j-\max\{\sup_Xu_j,0\}\leq0$. The subtracted constants tend to zero; decreasing energy continuity gives $E(v_j)\to E(u)$. Moreover,
\begin{equation*}
 e^{-kv_j}
 \leq C_k e^{-ku}\in L^1(\nu),\qquad
 e^{-kv_j}\longrightarrow e^{-ku}.
\end{equation*}
The constant is finite because the subtracted constants are bounded. Entrywise dominated convergence yields $H_k(v_j)\to H_k(u)$ and hence convergence of determinants. This proves \eqref{eq-gu-hilbert-bound}.
\end{proof}

The opposite energy comparison on the Bergman space is developed in \cite[Lemmas~7.7 and~7.8]{BBGZ13}; see also \cite[Proposition~4.4]{Zha24}. There, the quantized exhaustion controls the continuous energy exhaustion up to errors tending to zero. The next proof retains the rates furnished by \eqref{eq-gu-reference-expansion}. 

\begin{lem}
\label{lem-gu-matrix-energy}
If $H\geq I_N$ or $\sup_Xu_H=0$, then $E_k(H)\leq0$ and
\begin{equation}
 E(u_H)\geq E_k(H)-\frac{C_X}{k}(-E_k(H))-\frac{C_X}{k^2}.
 \label{eq-gu-matrix-energy}
\end{equation}
\end{lem}
\begin{proof}
Put $H_t=\exp(t\log H)$ and $f=\partial_tu_{H_t}\vert_{t=0}$. By diagonalizing $H$, $u_{H_t}$ is a logarithm of a sum of exponentials; hence it is pointwise convex and its strip extension is a subgeodesic. Direct differentiation gives
\begin{equation}
 f=-\frac{s^*(\log H)s}{m\rho_m},\qquad
 \int_X f\frac{\rho_m}{N}d\nu=E_k(H),\qquad
 f\leq u_H-u_I.
 \label{eq-matrix-tangent}
\end{equation}
The integral uses \eqref{eq-evaluation-column}. If $H\geq I_N$, then $f\leq0$ and $E_k(H)\leq0$. If $\sup u_H=0$, the last inequality and \eqref{eq-gu-reference-expansion} give $f\leq C_X/k^2$. Furthermore $s^*H^{-1}s\leq N$, so $\tr H^{-1}\leq N$. The arithmetic--geometric mean inequality implies $\det H^{-1}\leq1$, again giving $E_k(H)\leq0$. In either case,
\begin{equation*}
 \int_X\abs{f}\frac{\rho_m}{N}d\nu
 =2\int_X f_+\frac{\rho_m}{N}d\nu-E_k(H)
 \leq-E_k(H)+C_X/k^2.
\end{equation*}
The fixed expansion also gives
$d\mu_{u_I}=(1+O(k^{-1}))(\rho_m/N)d\nu$ and $\abs{E(u_I)}\leq C_X/k^2$.
Convexity of energy along the matrix-induced subgeodesic, \cite[Lemma~7.2]{BBGZ13}, now implies
\begin{align*}
 E(u_H)&\geq E(u_I)+\int_X f\,d\mu_{u_I}\\
 &\geq E_k(H)-\frac{C_X}{k}\int_X\abs{f}\frac{\rho_m}{N}d\nu-\frac{C_X}{k^2}\\
 &\geq E_k(H)-\frac{C_X}{k}(-E_k(H))-\frac{C_X}{k^2}.
\end{align*}
The second line substitutes \eqref{eq-matrix-tangent} and compares the two reference measures; the third uses the preceding absolute-integral bound.
\end{proof}

\subsection{A uniform \texorpdfstring{$d_1$}{d1} Bergman approximation}
The two energy lemmas do not by themselves estimate a distance. The additional ingredient is entropy duality applied to a density of integral one. It bounds the positive part of the approximation error against the \emph{input} Monge--Amp\`ere measure. The maximum of the two potentials then converts this one-sided estimate and the energy error into $d_1$.

\begin{prop}
\label{prop-full-approximation}
For every $B\geq0$, every $u\in\cK_B$, and every $k\geq k_0$ (the fixed reference threshold chosen above),
\begin{equation}
 d_1(u,\beta_ku)\leq\frac{C_X(1+B)}k,\qquad
 E(u)-E_k(H_k(u))\leq\frac{C_X(1+B)}k.
 \label{eq-full-approximation}
\end{equation}
\end{prop}
\begin{proof}
The density
\begin{equation}
 b_u(x)=\frac{s(x)^*H_u^{-1}s(x)}N e^{-ku(x)}
 \quad \text{satisfies}\quad
 \int_Xb_ud\nu=\frac{\tr(H_u^{-1}H_u)}N=1.
 \label{eq-approximation-density}
\end{equation}
Let $Z_u=\{u=-\infty\}$. Since $u\in L^1(\nu)$, $\nu(Z_u)=0$, and finite entropy gives $\mu_u\ll\nu$, hence $\mu_u(Z_u)=0$. On $X\setminus Z_u$, the density is positive and finite. Taking its logarithm and using \eqref{eq-bergman-map} gives the real-valued identity
\begin{equation*}
 m(\beta_ku-u)=\log b_u-\ell u.
\end{equation*}
Thus this particular logarithmic identity also holds $\mu_u$-almost everywhere. Set $b_u=1$ on $Z_u$, without changing any integral. Since $e^{(\log b_u)_+}=\max\{1,b_u\}\leq1+b_u$, Lemma~\ref{lem-measurable-entropy} gives $\int(\log b_u)_+d\mu_u\leq B+\log2$. The pointwise bound $m(\beta_ku-u)_+\leq(\log b_u)_++\ell(-u)$ then yields
\begin{equation}
 \int_X(\beta_ku-u)_+d\mu_u
 \leq\frac{B+\log2+\ell\int_X(-u)d\mu_u}{m}
 \leq\frac{C_X(1+B)}k.
 \label{eq-positive-error}
\end{equation}
The last step uses \eqref{eq-gu-entropy-control} and $m\geq k$. Integrability follows from these two estimates, not from absolute continuity alone.

Equations \eqref{eq-gu-hilbert-bound} and \eqref{eq-energy-budget} imply
$-E_k(H_u)\leq C_X(1+B)$ and
$E(u)-E_k(H_u)\leq C_X(1+B)/k$.
Since $H_u\geq I_N$, the estimate \eqref{eq-gu-matrix-energy} further gives
\begin{equation*}
 E(u)-E(\beta_ku)
 =[E(u)-E_k(H_u)]+[E_k(H_u)-E(u_{H_u})]
 \leq C_X(1+B)/k.
\end{equation*}
Finally set $v=\max\{u,\beta_ku\}$. Ordered distances \eqref{eq-gu-ordered} and concavity \eqref{eq-gu-concavity} give
\begin{align*}
 d_1(u,\beta_ku)
 &\leq 2[E(v)-E(u)]+E(u)-E(\beta_ku)\\
 &\leq2\int_X(\beta_ku-u)_+d\mu_u+C_X(1+B)/k\\
 &\underalign{\eqref{eq-positive-error}}{\leq}C_X(1+B)/k.
\end{align*}
This proves \eqref{eq-full-approximation}. Whenever an endpoint integral is needed later, we use \eqref{eq-d1-i1} to deduce $I_1(u,\beta_ku)\leq C_X(1+B)/k$.
\end{proof}

\section{Gaussian sections: the probability law and the potential}
\label{sec-Gaussian}
In this section we replace a full basis by $n+1$ Gaussian sections without changing the order of the error. 
The use of rotationally invariant random sections and logarithms of their norms goes back to Shiffman--Zelditch \cite[Lemma~3.1 and Section~4]{SZ99}. Their expected-zero-current calculation explains why averaging a logarithmic norm yields a Fubini--Study form. Our purpose is different: we keep $n+1$ sections and estimate an \emph{absolute} $d_1$ error, uniformly in the Hilbert form. The second endpoint measure also depends on these random sections, so an expected-current identity alone would not suffice.

Let $Z=(Z_{ij})_{1\leq i\leq n+1,\,1\leq j\leq N}$ be a random matrix. Its entries are independent complex random variables
\begin{equation*}
 Z_{ij}=X_{ij}+\sqrt{-1}Y_{ij},\qquad
 X_{ij},Y_{ij}\sim\cN(0,1/2).
\end{equation*}
All the real variables here are independent. Each pair $(X_{ij},Y_{ij})$ has density $\pi^{-1}e^{-x^2-y^2}dx\,dy$ on $\R^2$. Identify an $(n+1)\times N$ coefficient matrix with its $(n+1)N$ entries in $\mathbb C^{(n+1)N}$. For each positive Hermitian form $H$, define the probability $P_H$ on this space by the density
\begin{equation}
 dP_H(G)=\frac{(\det H)^{n+1}}{\pi^{(n+1)N}}
 e^{-\tr(GHG^*)}
 \prod_{i=1}^{n+1}\prod_{j=1}^N dx_{ij}\,dy_{ij},
 \qquad G_{ij}=x_{ij}+\sqrt{-1}y_{ij}.
 \label{eq-gaussian-law}
\end{equation}
This is equivalently the distribution of $G=ZH^{-1/2}$, since $Z=GH^{1/2}$ has real Jacobian $(\det H)^{n+1}$. More explicitly,
\begin{equation*}
 G_{ij}=\sum_{k=1}^NZ_{ik}(H^{-1/2})_{kj},\qquad
 \E_{P_H}[\overline{G_{ij}}G_{ik}]=(H^{-1})_{jk}.
\end{equation*}
Different rows are functions of disjoint independent rows of $Z$ and remain independent under $P_H$. Components in the same row are mixed by $H^{-1/2}$ and generally correlated; they are independent when $H$ is diagonal. In particular $\E_{P_H}(g_i^*g_i)=H^{-1}$ for the row vector $g_i=(G_{i1},\ldots,G_{iN})$.

The rows of $G$ give $n+1$ sections $\sum_jG_{ij}s_j$. Define their potential by
\begin{equation}
 u_G(x)=\frac1m\log\frac{\norm{Gs(x)}^2}{(n+1)N}
 =\frac1m\log
 \frac{\sum_{i=1}^{n+1}\abs{\sum_{j=1}^NG_{ij}s_j(x)}_{h_0^m}^2}{(n+1)N}.
 \label{eq-random-potential}
\end{equation}

\begin{lem}
\label{lem-gaussian-compression}
For $P_H$-almost every $G$, the sections in \eqref{eq-random-potential} have no common zero; $u_G$ is then smooth and semipositive. For every $H>0$,
\begin{equation}
 \int_{\mathbb C^{(n+1)N}}d_1(u_G,u_H)dP_H(G)\leq C_n/k.
 \label{eq-gaussian-compression}
\end{equation}
Consequently, for $u\in\cK_B$,
\begin{equation}
 \int_{\mathbb C^{(n+1)N}}d_1(u,u_G)dP_{H_u}(G)
 \leq C_X(1+B)/k.
 \label{eq-random-approximation}
\end{equation}
\end{lem}
\begin{proof}
\emph{Step 1: no common zero.}
For each $x$, the nonzero evaluation vector $s(x)$ makes $Gs(x)=0$ a system of $n+1$ independent complex linear equations. The incidence set of pairs $(G,x)$ therefore has dimension
\begin{equation*}
 (n+1)N-(n+1)+n=(n+1)N-1.
\end{equation*}
Its image in matrix space is a proper closed algebraic set: $X$ is projective and projection from its product is proper. This image has Lebesgue measure zero, hence $P_H$ measure zero by \eqref{eq-gaussian-law}. On its complement,
\begin{equation}
 m\omega_{u_G}
 =\ddc\log\frac{\norm{G\sigma}^2}{(n+1)N}.
 \label{eq-random-local-curvature}
\end{equation}

\emph{Step 2: normalize a holomorphic lift without changing the calculation.}
Fix a point with local coordinates $z_1,\ldots,z_n$ centered there. Under $P_H$ write $G=ZH^{-1/2}$ and put $F=H^{-1/2}\sigma$, so $G\sigma=ZF$. A constant unitary change in $F$ can rotate $F(0)/\norm{F(0)}$ to the first coordinate vector. Right multiplication of $Z$ by its inverse unitary leaves the law of $Z$ unchanged. After a nonzero constant scaling assume $F(0)=e_1$. Finally multiply $F$ by a nonvanishing holomorphic function $f$ satisfying
\begin{equation*}
 f(0)=1,\qquad \partial_if(0)=-e_1^*\partial_iF(0).
\end{equation*}
For example the exponential of the corresponding linear polynomial has this property. The modified lift satisfies
\begin{equation*}
 F(0)=e_1,\qquad e_1^*\partial_iF(0)=0\quad(1\leq i\leq n).
\end{equation*}
These operations preserve $m\omega_{u_H}$ and $m\omega_{u_G}$, the curvature forms of the two metrics $h_0^me^{-mu_H}$ and $h_0^me^{-mu_G}$ on $mL$: multiplication of the lift by $f$ adds only $\ddc\log\abs{f}^2=0$. They preserve $\norm{ZF}^2/\norm{F}^2$ as well, so both the potential difference and its endpoint measures are unchanged.

Let $J$ be the $(N-1)\times n$ matrix with columns formed by the last $N-1$ components of $\partial_iF(0)$. Differentiating $\log\norm{F}^2$ explicitly gives
\begin{equation*}
 \partial_i\bar\partial_j\log\norm{F}^2(0)
 =(\partial_jF)^*\partial_iF
 -((\partial_jF)^*F)(F^*\partial_iF)
 =(J^*J)_{ji}.
\end{equation*}
The subtracted term vanishes by the normalization. By \eqref{eq-fs-local-curvature}, in the convention $\ddc=\frac{\sqrt{-1}}{2\pi}\partial\bar\partial$, the matrix of $m\omega_{u_H}$ is precisely this Hermitian Gram matrix.

\emph{Step 3: compute the random curvature matrix conditionally.}
Write $Z=[g,W]$, where $g\in\mathbb C^{n+1}$ is its first column and $W$ consists of the other columns. They are independent standard complex Gaussians. At the chosen point, $ZF=g$ and its derivative matrix is $WJ$. Differentiating the logarithm in \eqref{eq-random-local-curvature} gives the Hermitian matrix
\begin{equation}
 \frac{J^*W^*\left(I_{n+1}-gg^*/\norm{g}^2\right)WJ}{\norm{g}^2}
 \quad\text{for }m\omega_{u_G}.
 \label{eq-random-curvature-matrix}
\end{equation}
This is the orthogonal projection of the derivatives onto $g^\perp$, divided by the squared norm of the value. The event $g=0$ has probability zero.

Conditional on $g$, choose an orthonormal basis of $g^\perp$. In that basis the projected matrix $W$ has $n$ rows of independent standard complex Gaussians: left unitary invariance gives this assertion, and independence of $W$ and $g$ permits conditioning. To compute its determinant moment, recall that the Cauchy--Binet identity is
\begin{equation*}
 \det(J^*J)=\sum_{\substack{I\subset\{1,\ldots,N-1\}\\\abs{I}=n}}\abs{\det J_I}^2;
\end{equation*}
see \cite[Section~0.8.7]{HJ13}. Expanding the determinant of the product of the projected $W$ with $J$, unequal column-index sets have zero mixed expectation by independence and phase invariance. For an $n\times n$ standard Gaussian matrix $Z_n$,
\begin{equation*}
 \E\abs{\det Z_n}^2
 =\sum_{\sigma,\tau\in\mathfrak S_n}
 \sgn(\sigma)\sgn(\tau)
 \prod_{i=1}^n\E[Z_{i,\sigma(i)}\overline{Z_{i,\tau(i)}}]
 =n!.
\end{equation*}
Each expectation in the product is a Kronecker delta, so only equal permutations contribute, giving $n!$ terms. More explicitly, let $A$ be the projected $n\times(N-1)$ standard Gaussian matrix on $g^\perp$. Cauchy--Binet and phase invariance give
\begin{align*}
 \E[\abs{\det(AJ)}^2\mid g]
 &=\sum_{I,K}\det J_I\overline{\det J_K}\,
       \E[\det A_I\overline{\det A_K}]\\
 &=n!\sum_I\abs{\det J_I}^2=n!\det(J^*J).
\end{align*}
The index sets have $n$ elements; if $I\ne K$, multiplying a column in their symmetric difference by a phase $e^{\sqrt{-1}\theta}$ changes the integrand but preserves its law, so its expectation is zero. Dividing by $\norm{g}^{2n}$, as required by \eqref{eq-random-curvature-matrix}, gives
\begin{equation}
 \E[\mu_{u_G}\mid g]
 =\frac{n!}{\norm{g}^{2n}}\mu_{u_H}.
 \label{eq-gaussian-top-form}
\end{equation}
All common differential-form factors, including $m^{-n}$, cancel from this ratio.

\emph{Step 4: the two absolute endpoint moments.}
At the fixed point,
\begin{equation*}
 u_G-u_H=\frac1m\log\frac{\norm{g}^2}{n+1}.
\end{equation*}
Each $\abs{g_i}^2$ has density $e^{-t}$ on $t>0$, by polar integration of $\pi^{-1}e^{-\abs{z}^2}dx\,dy$. If the sum of $q$ independent such variables has density $t^{q-1}e^{-t}/(q-1)!$, convolution with the next gives
\begin{equation*}
 \int_0^t\frac{s^{q-1}e^{-s}}{(q-1)!}e^{-(t-s)}ds
 =\frac{t^qe^{-t}}{q!}.
\end{equation*}
Thus $\norm{g}^2$ has density $t^ne^{-t}/n!$. For the random endpoint, the difference depends only on $g$, so conditioning in \eqref{eq-gaussian-top-form} gives, pointwise in $x$,
\begin{equation*}
 \E \big[\abs{u_G-u_H}\,d\mu_{u_G}(x)\mid g\big]
 =\frac1m\abs{\log\frac{\norm{g}^2}{n+1}}
    \frac{n!}{\norm{g}^{2n}}d\mu_{u_H}(x).
\end{equation*}
Integrating first in $g$ and then in $x$ is justified by nonnegative Tonelli, without assuming finiteness in advance. Since $\mu_{u_H}(X)=1$, we obtain
\begin{align}
 \E_{P_H}\int_X\abs{u_G-u_H}d\mu_{u_H}
 &=\frac1{mn!}\int_0^\infty
      \abs{\log\frac{t}{n+1}}t^ne^{-t}dt,\notag\\
 \E_{P_H}\int_X\abs{u_G-u_H}d\mu_{u_G}
 &=\frac1m\int_0^\infty
       \abs{\log\frac{t}{n+1}}e^{-t}dt.
 \label{eq-gaussian-endpoints}
\end{align}
In the second line the factor $n!t^{-n}$ in \eqref{eq-gaussian-top-form} cancels the gamma density's factor $t^n/n!$. Both integrals are finite: $\abs{\log t}$ is integrable at zero, and the exponential controls infinity. The sum of the two quantities in \eqref{eq-gaussian-endpoints} is exactly $\E_{P_H}I_1(u_G,u_H)$, by the definition recalled in Theorem~\ref{thm-darvas-inputs}. The inequality $d_1\leq I_1$ and $m\geq k$ prove \eqref{eq-gaussian-compression}. Applying it with $H=H_u$, and then the triangle inequality and \eqref{eq-full-approximation}, proves \eqref{eq-random-approximation}.
\end{proof}

\section{Determinantal sampling and integer interpolation}
\label{sec-sampling}
Equation \eqref{eq-random-approximation} approximates $u$ by a random potential with coefficient law $P_{H_u}$. This law varies with $u$, so it does not yet supply a finite cover. We will construct a single reference probability for each coarse metric center $v$ and compare every nearby input law with it. The construction has three components: a balanced sample of points, integer-valued observations of $\beta_ku-v$, and Gaussian coefficients.

\subsection{Balanced Hilbert forms and the variational gap}

Fixed-measure balancing has its antecedent in the projective
center-of-mass construction of Bourguignon--Li--Yau
\cite[Section~2 (c)--(d)]{BLY94}.
Donaldson \cite[Section~2.2 and Proposition~3]{Don09} gives a 
fixed-measure formulation for more general Radon measures.
Here the probability measure is prescribed, rather than taken
to be the metric's own Fubini--Study volume.
We use this formulation and the variational treatment of
\cite[Section~7.2, Lemma~7.4]{BBGZ13}, with the sign convention
in which the functional is minimized.
The quantitative task is to compare the minimizer with the Hilbert
form $H_u$ supplied by Bergman approximation; these are generally
different Hermitian forms.


\begin{lem}
\label{lem-balanced-gap}
For $u\in\cK_B$, the functional on positive Hermitian forms
\begin{equation}
 F_{\mu_u}(H)=\int_Xu_Hd\mu_u-E_k(H)
 =\frac1m\int_X\log\frac{s^*H^{-1}s}{N}d\mu_u
   +\frac1{mN}\log\det H
 \label{eq-balanced-functional}
\end{equation}
has a minimizer $\widehat H_u$ modulo positive scalar multiplication. It satisfies
\begin{equation}
 \int_X\frac{ss^*}{s^*\widehat H_u^{-1}s}d\mu_u
   =\frac{\widehat H_u}{N},
 \qquad
 0\leq F_{\mu_u}(H_u)-F_{\mu_u}(\widehat H_u)
 \leq\frac{C_X(1+B)}k.
 \label{eq-balanced-gap}
\end{equation}
\end{lem}
\begin{proof}
\emph{Existence.}
Replacing $H$ by $cH$ subtracts $(\log c)/m$ from both $u_H$ and $E_k(H)$, so $F_{\mu_u}$ is scale invariant. Restrict to $\det H=1$. For a reference-unit section $t$, orthonormality gives $\sup_X\abs{t}_{h_0^m}^2\geq1$. The potential
\begin{equation*}
 \frac1m\left(\log\abs{t}_{h_0^m}^2-\sup_X\log\abs{t}_{h_0^m}^2\right)
\end{equation*}
is normalized $\omega$-psh. Applying \eqref{eq-gu-entropy-control}, including its divisor singularities, gives
$\int_X\log\abs{t}_{h_0^m}^2d\mu_u\geq-mC_X(1+B)$.
Take $t$ corresponding to a unit eigenvector for the largest eigenvalue of $H^{-1}$. The inequality
$s^*H^{-1}s\geq\lambda_{\max}(H^{-1})\abs{t}^2$ implies
\begin{equation*}
 F_{\mu_u}(H)\geq\frac1m\log\lambda_{\max}(H^{-1})
                   -C_X(1+B)-\frac{\log N}{m}.
\end{equation*}
At fixed $k$, bounded sublevels therefore bound the largest eigenvalue of $H^{-1}$. Since $\det H^{-1}=1$, they also bound its smallest eigenvalue away from zero. Thus sublevels are compact in the positive cone. On such a compact set $s^*H^{-1}s$ is bounded above and below by fixed positive multiples of $\rho_m$, so $F_{\mu_u}$ is continuous and attains its minimum.

\emph{The balanced equation.}
For any Hermitian variation $\dot H$, differentiation of \eqref{eq-balanced-functional} gives
\begin{equation*}
 m\,dF_{\mu_u}\vert_H(\dot H)
 =-\int_X\frac{s^*H^{-1}\dot H H^{-1}s}{s^*H^{-1}s}d\mu_u
    +\frac1N\tr(H^{-1}\dot H).
\end{equation*}
At a minimizer this vanishes for every Hermitian $\dot H$; scale invariance permits unrestricted variations. By cyclicity of trace the coefficient paired with $\dot H$ is
\begin{equation*}
 -H^{-1}\left(\int_X\frac{ss^*}{s^*H^{-1}s}d\mu_u\right)H^{-1}
 +\frac1NH^{-1}=0.
\end{equation*}
Indeed a Hermitian matrix with zero trace pairing against every Hermitian matrix is zero. Multiplying this equality on the left and right by $H$ gives
\begin{equation*}
 \int_X\frac{ss^*}{s^*H^{-1}s}d\mu_u=H/N,
\end{equation*}
which proves the first identity in \eqref{eq-balanced-gap}. Equivalently, after conjugating by $\widehat H_u^{-1/2}$,
\begin{equation*}
 \int_X\frac{\widehat H_u^{-1/2}ss^*\widehat H_u^{-1/2}}
 {s^*\widehat H_u^{-1}s}d\mu_u=I_N/N.
\end{equation*}
This is the fixed-measure balanced condition for the normalized evaluation vectors.

\emph{Comparison with $H_u$.}
Rescale $\widehat H_u$ so $\sup_Xu_{\widehat H_u}=0$. Minimality against $I_N$ and \eqref{eq-gu-reference-expansion} give
\begin{equation*}
 \int_Xu_{\widehat H_u}d\mu_u-E_k(\widehat H_u)
 \leq\int_Xu_Id\mu_u=O(k^{-2}).
\end{equation*}
Consequently $-E_k(\widehat H_u)\leq C_X(1+B)$ by \eqref{eq-gu-entropy-control}. Lemma~\ref{lem-gu-matrix-energy} now implies
$E_k(\widehat H_u)-E(u_{\widehat H_u})\leq C_X(1+B)/k$.
By concavity of $E$, we obtain
\begin{align}
 F_{\mu_u}(H_u)-F_{\mu_u}(\widehat H_u)
 &=\int_X(\beta_ku-u)d\mu_u+\int_X(u-u_{\widehat H_u})d\mu_u\notag\\
 &\hspace{8mm}-E_k(H_u)+E_k(\widehat H_u)\notag\\
 &\underalign{\eqref{eq-gu-concavity}}{\leq}
 \int_X(\beta_ku-u)_+d\mu_u+[E(u)-E_k(H_u)]\notag\\
 &\hspace{8mm}+[E_k(\widehat H_u)-E(u_{\widehat H_u})]\notag\\
 &\leq C_X(1+B)/k.
 \label{eq-balanced-cancellation}
\end{align}
The final line uses \eqref{eq-positive-error}, \eqref{eq-full-approximation}, and the preceding bound, respectively. Concavity at $u$ justifies the middle line, and minimality gives nonnegativity. Thus \eqref{eq-balanced-cancellation} proves the remaining assertion of \eqref{eq-balanced-gap}.
\end{proof}

\subsection{A determinantal probability with prescribed marginals}
Projection determinantal processes attach to a finite-dimensional Hilbert space a probability density given by a squared evaluation determinant; see \cite[Definition~3 and Section~2]{HKPV06}. Berman develops this construction for spaces of holomorphic sections and relates their asymptotic densities to Monge--Amp\`ere measures \cite{Ber14, Berm18}. We need a finite-$k$ statement instead: balancing makes each point marginal \emph{exactly} $\mu_u$, not just asymptotic to it. This exact identity is what will control the cost of the integer observations.

For a positive form $H$ define the normalized evaluation Gram matrix on $X^N$ by
\begin{equation}
 (K_H(x_1,\ldots,x_N))_{ij}
 =\frac{s(x_i)^*H^{-1}s(x_j)}
 {\sqrt{(s(x_i)^*H^{-1}s(x_i))(s(x_j)^*H^{-1}s(x_j))}}.
 \label{eq-gram-definition}
\end{equation}
It has diagonal one and determinant between zero and one by Hadamard's inequality. Unit-frame changes conjugate it by a diagonal unitary matrix, leaving its determinant unchanged. We use \emph{ordered} point tuples.

\begin{lem}
\label{lem-determinantal-sampling}
For every $u\in\cK_B$ there is a probability $\Gamma_u$ on $X^N$ such that, for each coordinate projection $p_i:X^N\to X$,
\begin{equation}
 (p_i)_*\Gamma_u=\mu_u,\qquad
 \int_{X^N}\sum_{i=1}^N f(x_i)d\Gamma_u
      =N\int_X f\,d\mu_u
 \label{eq-point-marginals}
\end{equation}
for every nonnegative measurable $f$. Moreover,
\begin{equation}
 \Ent(\Gamma_u\mid\nu^{\otimes N})\leq N(1+B),\qquad
 \int_{X^N}-\log\det K_{H_u}\,d\Gamma_u\leq C_X(1+B)N.
 \label{eq-determinantal-cost}
\end{equation}
The evaluation matrix $S=[s(x_1),\ldots,s(x_N)]$ is invertible almost surely under $\Gamma_u$.
\end{lem}
\begin{proof}
\emph{The density and its normalizer.}
Use the balanced form from Lemma~\ref{lem-balanced-gap} and set
\begin{equation*}
 z(x)=\frac{\widehat H_u^{-1/2}s(x)}
                  {\sqrt{s(x)^*\widehat H_u^{-1}s(x)}}.
\end{equation*}
Then $\norm{z(x)}=1$, and \eqref{eq-balanced-gap} says
$\int_Xz_a\overline{z_b}\,d\mu_u=\delta_{ab}/N$.
Define
\begin{equation}
 d\Gamma_u=\frac{N^N}{N!}\det K_{\widehat H_u}\,
               d\mu_u(x_1)\cdots d\mu_u(x_N).
 \label{eq-determinantal-law}
\end{equation}
To verify the constant, expand the squared determinant:
\begin{align*}
 &\int_{X^N}\abs{\det[z(x_1),\ldots,z(x_N)]}^2d\mu_u^{\otimes N}\\
 &\quad=\sum_{\sigma,\tau\in\mathfrak S_N}
 \sgn(\sigma)\sgn(\tau)
 \prod_{i=1}^N\int_Xz_{\sigma(i)}(x)\overline{z_{\tau(i)}(x)}d\mu_u(x)\\
 &\quad=\sum_{\sigma\in\mathfrak S_N}N^{-N}=\frac{N!}{N^N}.
\end{align*}
The first equality is the Leibniz formula and finite-product integration; the second uses the balanced identity \eqref{eq-balanced-gap}, namely $\int z_a\overline{z_b}d\mu_u=\delta_{ab}/N$. Since $K_{\widehat H_u}$ is the Gram matrix of these columns, \eqref{eq-determinantal-law} has total mass one.

\emph{The exact marginals.}
Fix $x_1$. A unitary rotation sends $z(x_1)$ to $e_1$ and preserves the second-moment matrix. Expand the determinant along its first column and integrate the resulting $(N-1)\times(N-1)$ minor over $x_2,\ldots,x_N$. The same permutation computation gives
\begin{equation*}
 \int_{X^{N-1}}\det K_{\widehat H_u}(x_1,\ldots,x_N)
             \prod_{i=2}^Nd\mu_u(x_i)=\frac{(N-1)!}{N^{N-1}}.
\end{equation*}
Multiplying by $N^N/N!$ gives one. Thus for every bounded measurable $f$,
\begin{equation*}
 \int_{X^N}f(x_1)d\Gamma_u
 =\frac{N^N}{N!}\frac{(N-1)!}{N^{N-1}}\int_Xf\,d\mu_u
 =\int_Xf\,d\mu_u.
\end{equation*}
Symmetry gives each marginal; monotone convergence and summation give \eqref{eq-point-marginals}.

\emph{Entropy of the point law.}
The determinant is positive $\Gamma_u$-almost surely by its density. The function $-t\log t$ is bounded on $[0,1]$, 
so $\log\det K_{\widehat H_u}$ is integrable against $\Gamma_u$. 
Taking the logarithm of \eqref{eq-determinantal-law},
\begin{equation*}
 0\leq\Ent(\Gamma_u\mid\mu_u^{\otimes N})
 =\log\frac{N^N}{N!}+\int\log\det K_{\widehat H_u}\,d\Gamma_u
 \leq\log\frac{N^N}{N!}\leq N.
\end{equation*}
The last inequality follows from $\log N!\geq N\log N-N$, obtained by integrating $\log t$ from $1$ to $N$. Rearranging the nonnegativity inequality gives
$\int-\log\det K_{\widehat H_u}\,d\Gamma_u\leq N$.

Write $p=d\mu_u/d\nu$. Finite entropy implies $\log p\in L^1(\mu_u)$: on $p<1$, $-p\log p\leq1/e$, and the positive part is controlled by entropy. Since
\begin{equation*}
 \log\frac{d\Gamma_u}{d\nu^{\otimes N}}
 =\log\frac{d\Gamma_u}{d\mu_u^{\otimes N}}+\sum_{i=1}^N\log p(x_i)
 \quad\Gamma_u\text{-almost surely},
\end{equation*}
integration and \eqref{eq-point-marginals} give
\begin{equation*}
 \Ent(\Gamma_u\mid\nu^{\otimes N})
 =\Ent(\Gamma_u\mid\mu_u^{\otimes N})+N\Ent(\mu_u\mid\nu)
 \leq N(1+B).
\end{equation*}

\emph{Transfer to the Hilbert form $H_u$.}
When the evaluation matrix $S$ is invertible, taking determinants 
in \eqref{eq-gram-definition} gives
\begin{equation}
 \det K_H=
 \frac{(\det H)^{-1}\abs{\det S}^2}
 {\prod_{i=1}^N s(x_i)^*H^{-1}s(x_i)}.
 \label{eq-gram-determinant}
\end{equation}
Invertibility holds $\Gamma_u$-almost surely because its density vanishes when $\det K_{\widehat H_u}=0$. Taking the ratio of \eqref{eq-gram-determinant} for $H_u$ and $\widehat H_u$ cancels $\abs{\det S}^2$ and gives
\begin{align*}
 -\log\det K_{H_u}+\log\det K_{\widehat H_u}
 &=\log\frac{\det H_u}{\det\widehat H_u}
   +\sum_{i=1}^N\log\frac{s(x_i)^*H_u^{-1}s(x_i)}{s(x_i)^*\widehat H_u^{-1}s(x_i)}\\
 &=mN[E_k(\widehat H_u)-E_k(H_u)]
   +m\sum_{i=1}^N(\beta_ku-u_{\widehat H_u})(x_i).
\end{align*}
The second line uses \eqref{eq-gu-fs}. Integrating the sum by \eqref{eq-point-marginals} and combining its terms according to \eqref{eq-balanced-functional} gives
\begin{equation}
 \int-\log\det K_{H_u}\,d\Gamma_u
 =\int-\log\det K_{\widehat H_u}\,d\Gamma_u
  +mN[F_{\mu_u}(H_u)-F_{\mu_u}(\widehat H_u)].
 \label{eq-gram-transfer}
\end{equation}
Here $\log(s^*H_u^{-1}s/s^*\widehat H_u^{-1}s)$ is bounded on $X$ for the two fixed positive forms, so its integration is legitimate. Substituting \eqref{eq-balanced-gap} into \eqref{eq-gram-transfer}, and using the bounded ratio $m/k$, proves the second bound of \eqref{eq-determinantal-cost}.
\end{proof}

\subsection{The reference distribution: sample points, quantized observations, 
and Gaussian section coefficients} 
The sample space is the Borel product
\begin{equation*}
 \Omega_k=X^N\times\mathbb Z^N\times\mathbb C^{(n+1)N}.
\end{equation*}
Its coordinates are $(x,\alpha,G)$, with $x=(x_1,\ldots,x_N)$, $\alpha=(\alpha_1,\ldots,\alpha_N)$, and $G$ a numerical coefficient matrix. The last factor has complex dimension $(n+1)N$. The decoder is
\begin{equation}
 T_k(x,\alpha,G)=
 \begin{cases}u_G,&\text{if the sections specified by }G\text{ have no common zero},\\
 0,&\text{otherwise}.
 \end{cases}
 \label{eq-common-decoder}
\end{equation}
Here $u_G$ is \eqref{eq-random-potential}. The points and integers specify a reference distribution of coefficients; the decoded potential itself uses only $G$.

\begin{prop}
\label{prop-common-law}
For every $v\in\cK_B$ and $k\geq k_0$ there is a probability $Q_{v,k}$ on $\Omega_k$ such that for every $u\in\cK_B$ there is a probability $P_{u;v,k}$ satisfying
\begin{align}
 \E_{P_{u;v,k}}d_1(u,T_k)&\leq C_X(1+B)/k,
 \label{eq-common-error}\\
 \Ent(P_{u;v,k}\mid Q_{v,k})&\leq C_XN_k[1+B+kd_1(u,v)].
 \label{eq-common-cost}
\end{align}
\end{prop}
\begin{proof}
\emph{Step 1: relative entropy between coefficient laws.}
From \eqref{eq-gaussian-law}, for positive Hermitian $H,A$,
\begin{equation*}
 \log\frac{dP_H}{dP_A}(G)
 =(n+1)\log\frac{\det H}{\det A}
       +\tr(G(A-H)G^*).
\end{equation*}
Since $G=ZH^{-1/2}$ and $\E(Z^*Z)=(n+1)I_N$, one has
$\E_{P_H}(G^*G)=(n+1)H^{-1}$. Taking expectations gives
\begin{equation}
 \Ent(P_H\mid P_A)
 =(n+1)\left[\log\frac{\det H}{\det A}
                  +\tr(AH^{-1})-N\right].
 \label{eq-gaussian-kl}
\end{equation}
The logarithmic determinant is the term to be estimated by determinantal sampling.

\emph{Step 2: define $Q_{v,k}$.}
Use the unit frames fixed in \eqref{eq-evaluation-column} to regard
$S(x)=[s(x_1),\ldots,s(x_N)]$ as a numerical $N\times N$ matrix. For an invertible $S(x)$ and finite $v(x_i)$, define
\begin{equation}
 A_{v,k}(x,\alpha)
 =S(x)\diag
 \left(\frac{e^{-mv(x_1)-\alpha_1}}{N},\ldots,
       \frac{e^{-mv(x_N)-\alpha_N}}{N}\right)S(x)^*.
 \label{eq-interpolation-matrix}
\end{equation}
This is a positive Hermitian precision matrix. In particular its inverse is
\begin{equation*}
 A_{v,k}(x,\alpha)^{-1}
 =(S(x)^{-1})^*
 \diag(Ne^{mv(x_1)+\alpha_1},\ldots,Ne^{mv(x_N)+\alpha_N})
 S(x)^{-1}.
\end{equation*}
Thus the associated Fubini--Study potential interpolates the prescribed values $v(x_i)+\alpha_i/m$ at the sampled points.
The scalar entries of $S$ are taken in the chosen unit frames. Changing any frame multiplies the corresponding column of $S$ by a phase; that phase cancels in \eqref{eq-interpolation-matrix}.

Set
\begin{equation*}
 \pi(a)=\frac{1-e^{-1}}{1+e^{-1}}e^{-\abs{a}},\qquad a\in\mathbb Z.
\end{equation*}
Define $Q_{v,k}$ by the explicit integration formula
\begin{equation}
 \int_{\Omega_k}\Phi\,dQ_{v,k}
 =\int_{X^N}\sum_{\alpha\in\mathbb Z^N}
       \left(\prod_{i=1}^N\pi(\alpha_i)\right)
       \int\Phi(x,\alpha,G)dP_{A_{v,k}(x,\alpha)}(G)
       d\nu^{\otimes N}(x)
 \label{eq-reference-product-law}
\end{equation}
for nonnegative measurable $\Phi$. Thus points are first sampled from $\nu^{\otimes N}$, integers independently from $\pi$, and coefficients from the indicated conditional Gaussian law.

The exceptional configurations in \eqref{eq-interpolation-matrix} have $\nu^{\otimes N}$ measure zero. Evaluation functionals span the dual section space, so some evaluation configuration is invertible. Its determinant is a nonzero holomorphic section over connected $X^N$, with a zero set of volume zero. Also $v$ is finite almost everywhere. On this exceptional set define $A_{v,k}=I_N$. The chosen Borel frames and the continuous Gaussian density \eqref{eq-gaussian-law} make \eqref{eq-reference-product-law} a measurable probability kernel and hence a probability law.

\emph{Step 3: define the input law and its coefficient marginal.}
For $u\in\cK_B$ put
\begin{equation}
 \alpha_i^{u,v}(x)=
 \left\lceil m\bigl(\beta_ku(x_i)-v(x_i)\bigr)\right\rceil,\quad 1\leq i\leq N.
 \label{eq-integer-marks}
\end{equation}
These integers are finite $\Gamma_u$-almost surely because $\Gamma_u\ll\nu^{\otimes N}$. Use $\Gamma_u$ from Lemma~\ref{lem-determinantal-sampling} and define
\begin{equation}
 \int_{\Omega_k}\Phi\,dP_{u;v,k}
 =\int_{X^N}\int
       \Phi(x,\alpha^{u,v}(x),G)dP_{H_u}(G)d\Gamma_u(x).
 \label{eq-input-product-law}
\end{equation}
For every Borel coefficient set $D$, the coefficient marginal of \eqref{eq-input-product-law} is
\begin{equation}
 P_{u;v,k}(X^N\times\mathbb Z^N\times D)
 =\int_{X^N}P_{H_u}(D)d\Gamma_u=P_{H_u}(D).
 \label{eq-coefficient-marginal}
\end{equation}
Thus \eqref{eq-random-approximation}, \eqref{eq-common-decoder}, and \eqref{eq-coefficient-marginal} prove \eqref{eq-common-error}.

\emph{Step 4: the cost of the integer observations.}
Since $\abs{\lceil t\rceil}\leq\abs{t}+1$, the exact point marginals \eqref{eq-point-marginals} give
\begin{align}
 \int\sum_{i=1}^N\abs{\alpha_i^{u,v}}d\Gamma_u
 &\leq N+mN\int_X\abs{\beta_ku-v}d\mu_u\notag\\
 &\leq N+mN\left[\int_X\abs{\beta_ku-u}d\mu_u
                         +\int_X\abs{u-v}d\mu_u\right]\notag\\
 &\underalign{\eqref{eq-d1-i1}}{\leq}
 N+C_nmN[d_1(\beta_ku,u)+d_1(u,v)]\notag\\
 &\underalign{\eqref{eq-full-approximation}}{\leq}
 C_XN[1+B+kd_1(u,v)].
 \label{eq-mark-cost}
\end{align}
The second line is the triangle inequality pointwise; the third uses the definition of $I_1$ to bound each displayed endpoint integral. The final bound also uses $m/k\leq1+\ell/k_0$. Exact marginals, rather than independent reference points, are essential here.

\emph{Step 5: the conditional Gaussian cost.}
For the integers \eqref{eq-integer-marks}, upward rounding gives
\begin{equation*}
 s(x_i)^*H_u^{-1}s(x_i)
 \leq Ne^{mv(x_i)+\alpha_i^{u,v}}
 \leq e\,s(x_i)^*H_u^{-1}s(x_i).
\end{equation*}
Substitute \eqref{eq-interpolation-matrix}, and use \eqref{eq-gram-determinant} for the determinant. The two required bounds are
\begin{align}
 \tr(A_{v,k}H_u^{-1})
 &=\sum_{i=1}^N
 \frac{s(x_i)^*H_u^{-1}s(x_i)}{Ne^{mv(x_i)+\alpha_i^{u,v}}}\leq N,\notag\\
 \log\frac{\det H_u}{\det A_{v,k}}
 &=\sum_{i=1}^N\log
 \frac{Ne^{mv(x_i)+\alpha_i^{u,v}}}{s(x_i)^*H_u^{-1}s(x_i)}
      -\log\det K_{H_u}\notag\\
 &\leq N-\log\det K_{H_u}.
 \label{eq-rounded-gaussian-cost}
\end{align}
Consequently \eqref{eq-gaussian-kl} and \eqref{eq-determinantal-cost} imply
\begin{equation*}
 \int_{X^N}\Ent(P_{H_u}\mid P_{A_{v,k}(x,\alpha^{u,v})})d\Gamma_u(x)
 \leq C_X(1+B)N.
\end{equation*}

\emph{Step 6: combine the three costs.}
On the support of $P_{u;v,k}$, the likelihood ratio factors as
\begin{equation*}
 \frac{dP_{u;v,k}}{dQ_{v,k}}(x,\alpha^{u,v}(x),G)
 =\frac{d\Gamma_u}{d\nu^{\otimes N}}(x)
   \prod_{i=1}^N\frac1{\pi(\alpha_i^{u,v}(x))}
   \frac{dP_{H_u}}{dP_{A_{v,k}(x,\alpha^{u,v}(x))}}(G).
\end{equation*}
All conditional Gaussian densities are positive. Taking logarithms and integrating \eqref{eq-input-product-law} gives
\begin{align}
 \Ent(P_{u;v,k}\mid Q_{v,k})
 &=\Ent(\Gamma_u\mid\nu^{\otimes N})
   +\int\sum_{i=1}^N[-\log\pi(\alpha_i^{u,v})]d\Gamma_u\notag\\
 &\quad+\int\Ent(P_{H_u}\mid
                   P_{A_{v,k}(x,\alpha^{u,v})})d\Gamma_u.
 \label{eq-product-kl}
\end{align}
The first term is at most $N(1+B)$. The second is the sum of $\abs{\alpha_i^{u,v}}$ plus a fixed constant per coordinate, and is bounded by \eqref{eq-mark-cost}. The third was bounded in Step~5. Substitution proves \eqref{eq-common-cost}.

Finally, the set of basepoint-free coefficient matrices is open, and $G\mapsto u_G$ is locally continuous in the smooth topology there, hence in $d_1$ by \eqref{eq-d1-i1}. Its complement has zero measure under every $P_H$, and therefore under the mixture \eqref{eq-reference-product-law}. This proves measurability of the decoder with its specified value on the exceptional set. Decoder outputs are not asserted to have entropy at most $B$; the next section selects the final centers inside $\cK_B$.
\end{proof}

\section{From relative entropy to metric covers}
\label{sec-coding}
The following lemma is a direct consequence of Fano's inequality \cite[Example~II.4, equation~(11)]{Gun11} and the packing--covering argument. We give its elementary binary data-processing proof.

\begin{lem}
\label{lem-relative-entropy-cover}
Let $(M,d)$ be a metric space and $U\subset M$ nonempty. Let $(\Omega,\cF)$ be a measurable space, $Q$ a probability on it, and $T:\Omega\to M$ a map for which $\omega\mapsto d(u,T(\omega))$ is measurable for every $u\in U$. Suppose for each $u\in U$ there is a probability $P_u$ on $(\Omega,\cF)$ with
\begin{equation}
 \E_{P_u}d(u,T)\leq r,\qquad
 \Ent(P_u\mid Q)\leq L,\qquad r>0.
 \label{eq-abstract-code-hypothesis}
\end{equation}
Then $U$ has an internal cover by closed $4r$-balls with at most $\exp(2(L+\log2))$ centers.
\end{lem}
\begin{proof}
\emph{Step 1: a successful-decoding event has positive $Q$ mass.}
For $u\in U$ define
\begin{equation*}
 F_u=\{\omega\in\Omega:d(u,T(\omega))\leq2r\}\in\cF .
\end{equation*}
Markov's inequality applied to \eqref{eq-abstract-code-hypothesis} gives $P_u(F_u)\geq1/2$. To compare this with $Q(F_u)$, take any $P\ll Q$ and event $F\in\cF$. Write $f=dP/dQ$, $p=P(F)$, and $q=Q(F)$. For $0<q<1$, Jensen applied to the probability $Q\vert_F/q$ gives
\begin{equation*}
 \int_Ff\log f\,dQ
 =q\int_F f\log f\,\frac{dQ}{q}
 \geq q\left(\frac pq\right)\log\left(\frac pq\right)
 =p\log(p/q).
\end{equation*}
The same calculation on $\Omega\setminus F$ gives
\begin{align*}
 \Ent(P\mid Q)&\geq p\log(p/q)+(1-p)\log((1-p)/(1-q))\\
 &\geq p\log(1/q)-\log2.
\end{align*}
In the second line we use
$p\log p+(1-p)\log(1-p)\geq-\log2$ and
$-(1-p)\log(1-q)\geq0$. When $0<q<1$, the cases $p=0,1$ are included by $0\log0=0$. If $q=0$ or $q=1$, absolute continuity forces $p=q$, and the binary relative entropy is zero. In our application $p\geq1/2$, so $q=0$ is impossible and $q=1$ already satisfies the desired bound. Applying the displayed inequality when $0<q<1$ gives
\begin{equation}
 Q(F_u)\geq\exp[-2(L+\log2)].
 \label{eq-success-mass}
\end{equation}

\emph{Step 2: packing implies disjoint events.}
If $u_1,\ldots,u_I$ have pairwise distances greater than $4r$, then $F_{u_i}\cap F_{u_j}=\varnothing$ for $i\ne j$. Otherwise an $\omega$ in the intersection would give
\begin{equation*}
 d(u_i,u_j)\leq d(u_i,T(\omega))+d(T(\omega),u_j)\leq4r,
\end{equation*}
a contradiction. Equation \eqref{eq-success-mass} therefore implies
\begin{equation*}
 1\geq Q\left(\bigcup_{i=1}^IF_{u_i}\right)
   =\sum_{i=1}^IQ(F_{u_i})
   \geq I\,e^{-2(L+\log2)}.
\end{equation*}
Thus $I\leq e^{2(L+\log2)}$.

\emph{Step 3: select centers in $U$.}
Choose $u_1\in U$. If the closed balls about $u_1,\ldots,u_j$ do not cover $U$, choose
\begin{equation*}
 u_{j+1}\in U\setminus\bigcup_{i=1}^j\overline B_d(u_i,4r).
\end{equation*}
Each chosen point is more than $4r$ from its predecessors, so Step~2 forces this procedure to terminate after at most $e^{2(L+\log2)}$ choices. At termination the displayed complement is empty. These are the required internal centers; compactness of $U$ was not assumed.
\end{proof}


\begin{prop}
\label{prop-local-cover}
There is a geometric constant $a>0$ such that, for every $v\in\cK_B$, $k\geq k_0$, and nonempty
$U\subset\cK_B\cap\overline B_{d_1}(v,\delta)$, the set $U$ has an internal closed-ball cover of radius $a(1+B)/k$ and logarithmic cardinality at most
\begin{equation}
 C_XN_k(1+B+k\delta).
 \label{eq-local-cover-cost}
\end{equation}
\end{prop}
\begin{proof}
Fix a geometric constant $C_X>0$ valid in \eqref{eq-common-error}. In Lemma~\ref{lem-relative-entropy-cover}, take
\begin{equation*}
 M=\cE^1(X,\omega),\quad d=d_1,\quad
 Q=Q_{v,k},\quad T=T_k,\quad
 r=C_X(1+B)/k .
\end{equation*}
Equation \eqref{eq-common-cost} gives
$L=C_XN_k(1+B+k\delta)$ for all $u\in U$.
The lemma's closed-ball radius is $4r$, so fix $a=4C_X$ with this choice of constant. Its logarithmic cardinality is at most $2(L+\log2)$; the fixed additive term is absorbed into \eqref{eq-local-cover-cost} because $N_k\geq1$ and $1+B\geq1$.
\end{proof}

\begin{proof}[Proof of the upper bound in Theorem~\ref{thm-d1-main}]
We construct finite sets of centers $\mathscr {V}_j\subset\cK_B$ and define their cardinalities by $M_j:=\#\mathscr{V}_j$. Set $k_j=2^jk_0$ and let the required radius at level $j$ be $a(1+B)/k_j$.

\emph{Initial cover.}
The zero potential belongs to $\cK_B$. By \eqref{eq-gu-ordered} and \eqref{eq-energy-budget},
$d_1(u,0)=-E(u)\leq C_X(1+B)$ for all $u\in\cK_B$. Apply Proposition~\ref{prop-local-cover} with $v=0$, $\delta=C_X(1+B)$, $k=k_0$, and $U=\cK_B$. It yields a finite set $\mathscr V_0$ with
\begin{equation*}
 \cK_B\subset\bigcup_{v\in\mathscr{V}_0}
             \overline B_{d_1}(v,a(1+B)/k_0),\qquad
 \log M_0\leq C_X(1+B)k_0^n .
\end{equation*}

\emph{Refinement.}
For each $v\in\mathscr{V}_{j-1}$ define its part of the input class by
\begin{equation*}
 U_{j,v}:=\cK_B\cap
           \overline B_{d_1}(v,a(1+B)/k_{j-1}).
\end{equation*}
Apply the local proposition to $U_{j,v}$ at $k_j$. It gives internal centers $\mathscr{W}_{j,v}\subset U_{j,v}$ with radius $a(1+B)/k_j$ and
\begin{align*}
 \log\#\mathscr{W}_{j,v}
 &\leq C_XN_{k_j}\left(1+B+k_j\frac{a(1+B)}{k_{j-1}}\right)\\
 &\leq C_X(1+B)k_j^n.
\end{align*}
The last line uses $k_j/k_{j-1}=2$ and \eqref{eq-gu-reference-expansion}.
Set $\mathscr{V}_j=\bigcup_{v\in\mathscr V_{j-1}}\mathscr W_{j,v}$.
The sets $U_{j,v}$ cover the input class, so $\mathscr V_j$ is a cover at the new radius. Moreover
\begin{equation*}
 M_j\leq\sum_{v\in\mathscr{V}_{j-1}}\#\mathscr{W}_{j,v}
 \leq M_{j-1}\exp[C_X(1+B)k_j^n].
\end{equation*}
The centers remain inside $\cK_B$. Errors are measured directly at each level and are not summed along the sequence of centers.

\emph{Cardinality and final scale.}
Iterating the preceding inequality gives
\begin{equation}
 \log M_j\leq C_X(1+B)\sum_{i=0}^jk_i^n
 \leq\frac{C_X(1+B)k_j^n}{1-2^{-n}}.
 \label{eq-dyadic-count}
\end{equation}
Choose the least $j\geq0$ such that $k_j\geq2a(1+B)/\e$. Minimality, or the fixed initial level when $j=0$, gives $k_j\leq C_X(1+B)/\e$. The closed covering radius is at most $\e/2$, so this is an open $\e$-cover. Substitution into \eqref{eq-dyadic-count} proves
\begin{equation*}
 H_\e(\cK_B,d_1)
 \leq \log M_j
 \leq C_X(1+B)k_j^n
 \leq C_X(1+B)^{n+1}\e^{-n}.
\end{equation*}
Run the same construction with the input class $\cK_B^{\cH}$ to obtain smooth internal centers: the local proposition allows any subset and chooses representatives in that subset. No entropy bound on the Gaussian decoder outputs is needed. Proposition~\ref{prop-d1-lower} supplies the matching lower bound.
\end{proof}

\section{Toric reduction to convex and transport coding}
\label{sec-toric}

This section proves Theorem~\ref{thm-toric-main}. We use the symplectic-potential description developed by Abreu \cite[Section~2, Theorem~2.8]{Abr00} and Donaldson \cite[Section~3.1]{Don02}. 
The characterization of finite energy by integrability of the 
Legendre transform already appears in the real convex setting
of Berman--Berndtsson \cite[Proposition~2.9]{BB13}. 
For the systematic toric pluripotential formulation, we use
Coman--Guedj--Sahin--Zeriahi
\cite[Proposition~3.2, Theorem~3.6, and Proposition~3.9]{CGSZ19}. 
The metric identity and measure transport below are consequences of this established correspondence. We use them to express the entropy constraint as a quantitative convex-function covering problem.

\subsection{Toric potentials and the exact metric identity}

Let $(X,\omega_0)$ be a polarized toric K\"ahler manifold of complex dimension $n$, and let $P\subset\R^n$ be its Delzant moment polytope.  Put
\begin{equation*}
V:=\int_X\omega_0^n,
\qquad \lambda_P:=\frac{1}{\vol(P)}\mathds{1}_P\,dy.
\end{equation*}
Thus $\lambda_P$ is normalized Lebesgue measure on $P$.  On the open orbit, write $x\in\R^n$ for logarithmic coordinates, let $F_0(x)$ be the strictly convex potential of $\omega_0$, and let
\begin{equation*}
u_0(y):=F_0^*(y)=\sup_{x\in\R^n}\{\langle x,y\rangle-F_0(x)\} 
\end{equation*}
be its Legendre transform.  The reference moment map is $m_0=\nabla F_0:\R^n\to P^\circ$.

For a torus-invariant finite-energy potential $\varphi$, let $F(x)=F_0(x)+\varphi(x)$ on the open orbit. Its Legendre transform and inverse are
\begin{equation*}
 u_\varphi(y)=\sup_{x\in\R^n}\{\langle x,y\rangle-F(x)\},\qquad
 F(x)=\sup_{y\in P}\{\langle x,y\rangle-u_\varphi(y)\}.
\end{equation*}
We take the lower semicontinuous extension to $P$ from $P^\circ$ and set it to $+\infty$ off $P$. The cited finite-energy characterization says that $u_\varphi\in L^1(P)$; conversely such convex functions give toric $\cE^1$ potentials. Additive constants are fixed by
\begin{equation}
\int_Pu_\varphi\,d\lambda_P=0.
\label{eq-toric-normalization}
\end{equation}
Explicitly, in polytope notation,
\begin{equation*}
 \cU_0^1(P)=\left\{u:P\to\R\cup\{+\infty\}:
 \begin{array}{l}
 u\text{ is convex, finite on }P^\circ,\quad u\in L^1(P),\\
 u\vert_{\partial P}\text{ is the lower semicontinuous extension},\\
 \int_Pu\,dy=0
 \end{array}\right\}.
\end{equation*}

At every differentiability point of $u\in\cU_0^1(P)$ define
\begin{equation}
T_u:=\nabla u_0^*\circ\nabla u
=\nabla F_0\circ\nabla u:P^\circ\longrightarrow P^\circ,
\label{eq-toric-transport-map}
\end{equation}
and put $\sigma_u:=(T_u)_*\lambda_P$.

\begin{prop}
\label{prop-toric-measure-identities}
For a torus-invariant $\varphi\in\cE^1(X,\omega_0)$ and $u=u_\varphi$,
\begin{equation}
(m_0)_*\MA(\varphi)=\sigma_u,
\qquad
\Ent(\MA(\varphi)\mid(V^{-1}\omega_0^n))
=\Ent(\sigma_u\mid\lambda_P).
\label{eq-toric-entropy-identity}
\end{equation}
The entropy identity holds with values in $[0,+\infty]$. 
\end{prop}

\begin{proof}
First suppose that $\varphi$ is smooth and strictly K\"ahler.  Let $F:=F_0+\varphi$ and let $m_\varphi=\nabla F$.  The action--angle expression for the toric volume form gives
\begin{equation*}
(m_\varphi)_*\MA(\varphi)=\lambda_P,
\qquad (m_0)_*(V^{-1}\omega_0^n)=\lambda_P;
\end{equation*}
see \cite[Section~2]{Abr00}.  If $y=m_\varphi(x)=\nabla F(x)$, Legendre duality gives $\nabla u(y)=x$.  Therefore
\begin{equation*}
T_u(y)=\nabla F_0(\nabla u(y))=\nabla F_0(x)=m_0(x).
\end{equation*}
Equivalently, $m_0=T_u\circ m_\varphi$ on the open orbit.  Pushing forward $\MA(\varphi)$ proves the first equality in \eqref{eq-toric-entropy-identity}.

The map $m_0$ identifies the open orbit modulo the compact torus with $P^\circ$, and both reference measures give zero mass to the complementary lower-dimensional orbits and faces.  More explicitly, the two torus-invariant measures disintegrate over $P^\circ$ with the same normalized Haar probability on every compact-torus orbit.  Their Radon--Nikodym derivative is therefore the pullback of the derivative of their $m_0$-pushforwards.  Integrating its logarithm proves the second equality in \eqref{eq-toric-entropy-identity}, including the value $+\infty$.

For a general toric finite-energy potential, \cite[Corollary~3.7]{CGSZ19} already supplies the change-of-variables identity in this regularity. After division by total volume it reads
\begin{equation*}
 \int_X\chi(x)\,d\MA(\varphi)=\int_P\chi(\nabla u(y))\,d\lambda_P(y)
\end{equation*}
for nonnegative continuous functions $\chi$ of the logarithmic coordinate; the complement of the open orbit has zero non-pluripolar mass. Taking $\chi=h\circ\nabla F_0$ proves the first pushforward identity for every nonnegative continuous $h$ on $P$. The invariant disintegration proves the entropy equality directly, so no convergence of entropy values is required.
\end{proof}

Toric geodesics are affine in Legendre coordinates
\cite[Section~4, Theorem~3]{Guan99}.
Combined with Darvas's length-minimizing theorem
\cite[Theorem~3.5]{Dar15}, this gives the $d_1$--$L^1$ identity below.
The calculation is the $p=1$ counterpart of Guedj's $L^2$ distance
formula \cite[Proposition~4.3]{Gue14}: one integrates the absolute
velocity instead of its square.
The finite-energy extension uses
\cite[Proposition~3.9]{CGSZ19} and metric completion.


\begin{prop}
\label{prop-toric-d1-isometry}
If $u,v\in\cU_0^1(P)$ correspond to toric finite-energy potentials $\varphi_u,\varphi_v$, then
\begin{equation}
d_1(\varphi_u,\varphi_v)=\int_P\abs{u-v}\,d\lambda_P.
\label{eq-toric-d1-identity}
\end{equation}
\end{prop}

\begin{proof}
For smooth endpoints, the weak geodesic is toric and its Legendre transform is the affine segment
\begin{equation*}
u_t=(1-t)u+tv
\end{equation*}
by \cite[Section~4, Theorem~3]{Guan99}.
If $F_t=u_t^*$, differentiating
\begin{equation*}
u_t(y)+F_t(x)=\langle x,y\rangle,
\qquad y=\nabla F_t(x),
\end{equation*}
gives $\dot F_t(x)=-\dot u_t(y)$.  Since $(\nabla F_t)_*\MA(\varphi_t)=\lambda_P$, the normalized Mabuchi norm of the velocity is
\begin{align*}
\norm{\dot\varphi_t}_{1,\varphi_t}
&=\int_X\abs{\dot\varphi_t}\,\MA(\varphi_t)\\
&=\int_P\abs{\dot u_t}\,d\lambda_P
=\int_P\abs{v-u}\,d\lambda_P.
\end{align*}
Darvas's minimizing constant-speed theorem \cite[Theorem~3.5]{Dar15} now gives \eqref{eq-toric-d1-identity} for smooth endpoints.

For general endpoints, take the decreasing smooth approximants of Theorem~\ref{thm-smooth-decreasing} and average over the compact torus. Averaging preserves decrease and strict positivity, and their limit is the original invariant potential. They converge in $d_1$ by \cite[Theorem~4.17]{Dar15}. Their Legendre transforms increase to $u$ and $v$ by \cite[Lemma~2.1(iii)]{CGSZ19}. Integrability of the limits follows from \cite[Proposition~3.9]{CGSZ19}; monotone convergence after subtracting the first transform proves $L^1(P)$ convergence. Subtracting their means preserves this convergence and changes the corresponding potentials by convergent constants. Passing to the limit in the smooth identity proves \eqref{eq-toric-d1-identity}.
\end{proof}

Define the convex entropy class
\begin{equation}
\cU_B^1(P;u_0):=\{u\in\cU_0^1(P):
\Ent((T_u)_*\lambda_P\mid\lambda_P)\leq B\}.
\label{eq-convex-entropy-class}
\end{equation}
For the class in \eqref{eq-convex-entropy-class}, Propositions~\ref{prop-toric-measure-identities} and~\ref{prop-toric-d1-isometry} give the exact coding identity
\begin{equation}
H_\e(\cK_B^{\rm tor},d_1)
=H_\e(\cU_B^1(P;u_0),L^1(\lambda_P)).
\label{eq-toric-convex-coding}
\end{equation}
Thus the K\"ahler covering problem has been reduced, without a loss in constants, to an $L^1$ covering problem for normalized convex functions.

\subsection{From relative entropy to a Sobolev bound}

Choose primitive integral outward normals $a_r$ and write
$P=\{y:\langle a_r,y\rangle\leq b_r,\ 1\leq r\leq d_P\}$.
Set $\ell_r(y)=b_r-\langle a_r,y\rangle$, so that each $\ell_r$
is nonnegative on $P$ and vanishes on the corresponding facet.
Guillemin's canonical boundary term, recalled in
\cite[Theorem~2.5]{Abr00}, together with Abreu's characterization
of smooth symplectic potentials \cite[Theorem~2.8]{Abr00}, gives 
\begin{equation}
u_0(y)=\frac12\sum_{r=1}^{d_P}\ell_r(y)\log\ell_r(y)+h_0(y),
\qquad h_0\in C^\infty(P).
\label{eq-guillemin-boundary}
\end{equation}

The next lemma is the polytope counterpart of the uniform exponential integrability \eqref{eq-alpha-integrability}: logarithmic facet singularities of the reference gradient replace the singularities of normalized $\omega$-psh potentials.

\begin{lem}
\label{lem-toric-exponential-moment}
There is $a_0>0$, depending only on $(P,u_0)$, such that
\begin{equation}
\int_Pe^{a_0\abs{\nabla u_0}}\,d\lambda_P<\infty.
\label{eq-reference-gradient-exponential}
\end{equation}
\end{lem}

\begin{proof}
Equation \eqref{eq-guillemin-boundary} gives
\begin{equation*}
\abs{\nabla u_0(y)}\leq C_0+C_1\sum_{r=1}^{d_P}\abs{\log\ell_r(y)}.
\end{equation*}
A Delzant polytope is simple.  Near a boundary point, the active facet functions form part of an affine coordinate system, and the inactive facet functions are bounded away from zero.  Hence, on every member of a finite boundary cover,
\begin{equation*}
e^{a\abs{\nabla u_0(y)}}
\leq C\prod_{r\ \mathrm{active}}\ell_r(y)^{-aC_1}.
\end{equation*}
The right-hand side is locally integrable in facet coordinates when $aC_1<1$.  On a compact subset of $P^\circ$ it is bounded.  A sufficiently small $a_0$ therefore proves \eqref{eq-reference-gradient-exponential}, including at the corners.
\end{proof}

Combining this exponential moment with relative entropy gives a gradient bound, just as \eqref{eq-gu-entropy-control} and \eqref{eq-energy-budget} give the energy bound in the general K\"ahler setting. The transport identity supplies the change of variables in this counterpart.

\begin{lem}
\label{lem-toric-sobolev-bound}
If $u\in\cU_B^1(P;u_0)$, then $u\in W^{1,1}(P^\circ)$ and
\begin{equation}
\int_P\abs{\nabla u}\,d\lambda_P\leq C(1+B).
\label{eq-toric-gradient-bound}
\end{equation}
Set $p_*=1+1/(2n)>1$. After the normalization \eqref{eq-toric-normalization},
\begin{equation}
\norm{u}_{L^{p_*}(\lambda_P)}\leq C(1+B).
\label{eq-toric-sobolev-bound}
\end{equation}
\end{lem}

\begin{proof}
A convex function finite on $P^\circ$ is locally Lipschitz and differentiable almost everywhere.  At a differentiability point, $\nabla u_0$ and $\nabla u_0^*=\nabla F_0$ are inverse diffeomorphisms, so \eqref{eq-toric-transport-map} gives
\begin{equation}
\nabla u=\nabla u_0\circ T_u
\quad\text{almost everywhere on }P.
\label{eq-gradient-transport-identity}
\end{equation}
Put $\sigma_u=(T_u)_*\lambda_P$. The nonnegative function $a_0\abs{\nabla u_0}$ is Borel on $P^\circ$ and may be defined arbitrarily on $\partial P$, a $\lambda_P$-null and hence $\sigma_u$-null set. Lemma~\ref{lem-measurable-entropy} applies by \eqref{eq-reference-gradient-exponential}, without presupposing gradient integrability. It gives
\begin{align*}
a_0\int_P\abs{\nabla u}\,d\lambda_P
&\underalign{\eqref{eq-gradient-transport-identity}}{=}
a_0\int_P\abs{\nabla u_0}\,d\sigma_u\\
&\underalign{\eqref{eq-measurable-entropy}}{\leq}
B+\log\int_Pe^{a_0\abs{\nabla u_0}}\,d\lambda_P
\leq C(1+B).
\end{align*}
This proves \eqref{eq-toric-gradient-bound}.  Since $u\in L^1(P)$ and its almost-everywhere gradient is integrable, the local distributional identity extends across compact subsets of $P^\circ$, so $u\in W^{1,1}(P^\circ)$.  Boundary values lie on a null set.

We only need a noncritical Sobolev exponent, which permits one argument in all dimensions. For a smooth function $f$ on the convex polytope, integrate the fundamental theorem of calculus on segments from $x$ to $y$ and average in $y$. Polar integration about $x$, with radial range bounded by the diameter of $P$, gives
\begin{equation*}
 \abs{f(x)-\textstyle\int_P f\,d\lambda_P}
 \leq C_P\int_P\frac{\abs{\nabla f(z)}}{\abs{x-z}^{n-1}}dz.
\end{equation*}
Indeed, reversing the radial integrals in $\int_0^{R(\theta)}\rho^{n-1}\int_0^\rho\abs{\nabla f(x+t\theta)}dt\,d\rho$ bounds the outer integral by a fixed constant and leaves the displayed kernel. Convexity ensures that the segments stay in $P$. Its $L^{p_*}$ norm in $x$ is uniformly bounded in $z$, since $p_*(n-1)<n$. Minkowski's integral inequality therefore gives $\norm{f-\int f}_{L^{p_*}}\leq C_P\norm{\nabla f}_{L^1}$. Smooth approximation on a polytope extends this to $W^{1,1}$ functions. Applying it to $u$ with zero mean proves \eqref{eq-toric-sobolev-bound}.
\end{proof}

To apply the covering theorem for real-valued convex functions on $P$, we first approximate convex functions that may be infinite on $\partial P$.

\begin{lem}
\label{lem-toric-boundary-closure}
Let $1<r<\infty$ and let $u$ be convex on $P^\circ$ with $u\in L^r(P)$. Then, with $R=\norm{u}_{L^r(P)}$,
\begin{equation*}
 u\in\overline{\{v:P\to\R:\ v\text{ convex},\ \norm{v}_{L^r(P)}\leq R\}}^{\,L^r(P)}.
\end{equation*}
This includes functions with value $+\infty$ on part of $\partial P$.
\end{lem}

\begin{proof}
Fix $y_*\in P^\circ$, let $A_\delta y=(1-\delta)y+\delta y_*$, and define
\begin{equation*}
u_\delta(y):=(1-\delta)^{n/r}u(A_\delta y).
\end{equation*}
Since $A_\delta P\Subset P^\circ$, the function $u_\delta$ is finite, continuous, and convex on the closed polytope.  Changing variables $z=A_\delta y$ gives
\begin{align*}
\norm{u_\delta}_{L^r(dy)}^r
&=(1-\delta)^n\int_P\abs{u(A_\delta y)}^r\,dy\\
&=\int_{A_\delta P}\abs{u(z)}^r\,dz
\leq\norm{u}_{L^r(dy)}^r.
\end{align*}
The composition operators $f\mapsto f\circ A_\delta$ are uniformly bounded on $L^r$ for small $\delta$ and converge strongly to the identity, first on continuous functions and then by density. Thus $u_\delta\to u$ in $L^r$.
\end{proof}

For example, finite transport entropy alone does not make $u$ finite on the closed polytope. On the standard two-simplex with its standard reference symplectic potential $u_0=\frac12[y_1\log y_1+y_2\log y_2+(1-s)\log(1-s)]$, put $s=y_1+y_2$ and
\begin{equation*}
u=u_0+s^{-\alpha}-c,
\qquad 0<\alpha<1,
\end{equation*}
where $c$ gives zero mean. Then $u\in L^2(P)$ and $\nabla u\in L^1(P)$, but $u(0,0)=+\infty$. To check finite transport entropy, put $q(s)=\exp(-2\alpha s^{-\alpha-1})$ and $D(s)=1-s+sq(s)$. 
Solving $x_i=\frac12\log(y_i/(1-y_1-y_2))$ for $y$ gives
\begin{equation*}
 (\nabla u_0)^{-1}(x)
 =\frac{(e^{2x_1},e^{2x_2})}{1+e^{2x_1}+e^{2x_2}}.
\end{equation*}
Substituting
$\nabla u=\nabla u_0-\alpha s^{-\alpha-1}(1,1)$
into this formula, and differentiating the resulting map, yields 
\begin{equation*}
 T_u(y)=\frac{q(s)}{D(s)}y,\qquad
 \det DT_u=\frac{q(s)^2}{D(s)^3}
 \left(1+2\alpha(\alpha+1)(1-s)s^{-\alpha-1}\right).
\end{equation*}
The determinant follows by differentiating the radial map: its angular factor is $q/D$ and its radial factor is $(sq/D)'$. The radial map is increasing from zero to one, hence $T_u$ is an interior diffeomorphism. Near $s=0$, its negative log determinant is $4\alpha s^{-\alpha-1}+O(1+\abs{\log s})$, and it is bounded near the other faces. The change-of-variables formula gives $\Ent((T_u)_*\lambda_P\mid\lambda_P)=-\int_P\log\det DT_u\,d\lambda_P$. Area measure is a constant times $s\,ds\,dt$ in coordinates $y_1=st,y_2=s(1-t)$, and $\alpha<1$ makes the integral finite. This illustrates the need for Lemma~\ref{lem-toric-boundary-closure} before applying a theorem about real-valued convex functions on the closed polytope.

\begin{proof}[Proof of the upper bound in Theorem~\ref{thm-toric-main}]
Put $r=p_*=1+1/(2n)$. Lemma~\ref{lem-toric-sobolev-bound} puts $\cU_B^1(P;u_0)$ in the $L^r(\lambda_P)$ ball of radius $C(1+B)$. Lemma~\ref{lem-toric-boundary-closure} permits us to cover its closure using real-valued convex functions on $P$.

The convex-function entropy theorem \cite[Theorem~1.1(ii)]{GW17}, specialized to $p=1<r$, states that if $P$ is triangulated into a fixed number of $n$-simplices, then
\begin{equation}
H_\e(\{v\text{ convex on }P:
\norm{v}_{L^r(\lambda_P)}\leq R\},L^1(\lambda_P))
\leq C_P(R/\e)^{n/2}.
\label{eq-gao-wellner-specialization}
\end{equation}
The source \cite{GW17} uses unnormalized Lebesgue norms. For the unit $L^r(dy)$ 
ball, the logarithm of the number of $L^1(dy)$-balls of radius
$\delta>0$ needed to cover the class is bounded by
$C_P(\vol(P)^{1-1/r}/\delta)^{n/2}$. 
Our radius-$R$ ball has Lebesgue radius $R\vol(P)^{1/r}$, and our covering radius corresponds to $\e\vol(P)$. Scaling the source theorem therefore replaces its ratio by
\begin{equation*}
 \frac{\vol(P)^{1-1/r}R\vol(P)^{1/r}}
 {\e\vol(P)}=\frac R\e,
\end{equation*}
proving \eqref{eq-gao-wellner-specialization}. A finite cover of the real-valued convex class covers its closure after a small radius enlargement. Discard balls not meeting the target entropy class, choose a target member in each remaining ball, and double the radius to make centers internal. These changes cost only a geometric factor. Taking $R=C(1+B)$ and using \eqref{eq-toric-convex-coding} proves \eqref{eq-toric-upper-intro}.
\end{proof}

\subsection{The toric lower bound}

Choose a closed cube $Q\Subset P^\circ$. On a neighborhood of $Q$, write $H_0=D^2u_0$ and $A_0=H_0^{-1}$; there is $\kappa>0$ such that $H_0\geq\kappa I$. Fix a nonzero nonnegative function $\psi\in C_c^\infty((0,1)^n)$. Choose a cube of fixed side length $L>0$ inside $Q$ (here $L$ denotes this length, not the polarization). Divide it into $N=k^n$ cubes of side $h=L/k$, with lower-left corners $y_j$, and put
\begin{equation*}
\psi_{j,h}(y):=h^2\psi((y-y_j)/h).
\end{equation*}
For $\theta\in\mathscr{C}_N$ from Lemma~\ref{lem-constant-weight-code}, define
\begin{equation}
w_\theta:=a\sum_{j=1}^N\theta_j\psi_{j,h},
\qquad
u_\theta:=u_0+w_\theta-c_\theta,
\label{eq-toric-bumps}
\end{equation}
where $c_\theta$ gives zero $\lambda_P$-mean.  Every word has the same Hamming weight and every scaled bump has the same integral, so $c_\theta$ is independent of $\theta$.

At any point at most one bump is nonzero and $\norm{D^2\psi_{j,h}}_\infty=\norm{D^2\psi}_\infty$.  Hence $D^2u_\theta\geq\kappa I/2$ if $a\norm{D^2\psi}_\infty\leq\kappa/2$.  Moreover, $w_\theta$ vanishes near $\partial P$, so $u_\theta$ has the same boundary form \eqref{eq-guillemin-boundary} as $u_0$.  The converse part of \cite[Theorem~2.8]{Abr00} shows that $u_\theta$ is the symplectic potential of a smooth toric K\"ahler metric.  Consequently, $T_\theta:=T_{u_\theta}$ is a diffeomorphism of $P^\circ$ and equals the identity near $\partial P$.

\begin{lem}
\label{lem-toric-quadratic-kl}
There are $a_0,K>0$, depending only on $(P,u_0)$ and $\psi$, such that
\begin{equation}
\Ent((T_\theta)_*\lambda_P\mid\lambda_P)\leq Ka^2
\label{eq-toric-kl-bound}
\end{equation}
for $0<a\leq a_0$, uniformly in $h$, $N$, and $\theta$.
\end{lem}

\begin{proof}
For a compactly supported perturbation $w$, put $q=\nabla w$ and $R=D^2w$.  We shall apply the calculation to the perturbations in \eqref{eq-toric-bumps}.  On the fixed compact set $Q$, the function
\begin{equation*}
\mathscr{L}(y,q,R):=-\log\det\left[
D^2F_0(\nabla u_0(y)+q)(H_0(y)+R)\right]
\end{equation*}
is uniformly $C^2$ near $(q,R)=(0,0)$.  To compute its full first variation, including the variation of $D^2F_0$, set
\begin{equation*}
T_s(y):=\nabla F_0(\nabla u_0(y)+s\nabla w(y)).
\end{equation*}
Since $\nabla F_0\circ\nabla u_0$ is the identity,
\begin{equation*}
\frac{d}{ds}\bigg\vert_{s=0}T_s=A_0\nabla w.
\end{equation*}
The derivative of the logarithmic Jacobian is the divergence of this vector field.  Taylor's formula therefore gives the pointwise expansion
\begin{equation}
\mathscr{L}(y,\nabla w,D^2w)
=-\operatorname{div}(A_0\nabla w)
+O(\abs{\nabla w}^2+\abs{D^2w}^2),
\label{eq-toric-jacobian-expansion}
\end{equation}
with a constant uniform on $Q$.

The density of $(T_\theta)_*\lambda_P$ at $T_\theta(y)$ is $(\det DT_\theta(y))^{-1}$.  The change-of-variables formula gives
\begin{equation}
\Ent((T_\theta)_*\lambda_P\mid\lambda_P)
=-\int_P\log\det DT_\theta\,d\lambda_P.
\label{eq-toric-entropy-jacobian}
\end{equation}
The integral of the linear term in \eqref{eq-toric-jacobian-expansion} vanishes because $A_0\nabla w_\theta$ is compactly supported in $Q$.  For one scaled bump,
\begin{equation*}
\int_P\abs{\nabla(a\psi_{j,h})}^2\,dy\leq Ca^2h^{n+2},
\qquad
\int_P\abs{D^2(a\psi_{j,h})}^2\,dy\leq Ca^2h^n.
\end{equation*}
The supports are disjoint and $Nh^n=L^n$.  Integrating \eqref{eq-toric-jacobian-expansion} and using \eqref{eq-toric-entropy-jacobian} therefore gives
\begin{equation*}
\Ent((T_\theta)_*\lambda_P\mid\lambda_P)
\leq CNa^2(h^{n+2}+h^n)\leq Ka^2,
\end{equation*}
which proves \eqref{eq-toric-kl-bound}.
\end{proof}

\begin{proof}[Proof of the lower bound in Theorem~\ref{thm-toric-main}]
Choose a fixed $\eta_0>0$ so small that $\eta_0\leq a_0$ and $K\eta_0^2\leq1$, and put
\begin{equation*}
a:=\eta_0\min\{1,\sqrt B\}.
\end{equation*}
Lemma~\ref{lem-toric-quadratic-kl} gives $\Ent((T_\theta)_*\lambda_P\mid\lambda_P)\leq B$.  For distinct constant-weight codewords, disjoint support and $\psi\geq0$ give the exact formula
\begin{align}
\norm{u_\theta-u_{\theta'}}_{L^1(\lambda_P)}
&=\frac{a}{\vol(P)}d_H(\theta,\theta')h^{n+2}
\int_{(0,1)^n}\psi\,dy\notag\\
&\geq cah^2,
\label{eq-toric-bump-separation}
\end{align}
where the common normalization constant cancels and $d_H(\theta,\theta')\geq\gamma N$.  Given $0<\e\leq c_0a$, choose $k$ comparable to $\sqrt{a/\e}$ so that the right-hand side of \eqref{eq-toric-bump-separation} is greater than $2\e$.  Then $N=k^n\geq c(a/\e)^{n/2}$, and Lemma~\ref{lem-constant-weight-code} gives a $2\e$-packing with
\begin{equation*}
\log\#\mathscr{C}_N\geq cN
\geq c\left(\frac{\min\{1,\sqrt B\}}{\e}\right)^{n/2}.
\end{equation*}
The coding identity \eqref{eq-toric-convex-coding} now proves \eqref{eq-toric-lower-intro}.
\end{proof}

\section{Applications and Outlook}
\label{sec-interpretation}
\subsection{Worst-case code length}
A finite $\e$-cover is a codebook: encode an input by the index of a center within distance $\e$ and decode by returning that center. A fixed-length binary index needs $\lceil\log_2N_\e\rceil$ bits. Conversely, any fixed-length code with at most $2^b$ decoded centers and worst-case distortion less than $\e$ gives an external cover of that size. Recentring in the input set changes the radius by at most two. 
Hence Theorems~\ref{thm-l1-main}--\ref{thm-toric-main} determine,
up to constants and radius rescaling, the worst-case description
exponents: $n$ for all normalized potentials in the background
$L^1$ metric and for fixed positive relative-entropy sublevels
in $d_1$, and $n/2$ for the corresponding toric sublevels. 
In this sense, our results 
connect coding theory with K\"ahler geometry. 

Our research starts from seeking a quantitative relation between relative entropy and metric entropy. 
Relative entropy $B$ and metric entropy $H_\e$ actually play different roles. The first constrains each Monge--Amp\`ere probability relative to the fixed reference volume. The second counts distinguishable potentials. In the general $d_1$ argument, \eqref{eq-energy-budget} bounds energy by $C_X(1+B)$, while \eqref{eq-common-cost} controls the extra information required relative to a coarse metric center. At adjacent dyadic levels $k\delta=O(1+B)$, so each refinement has logarithmic cost $O((1+B)k^n)$. Summation in \eqref{eq-dyadic-count} gives $O((1+B)^{n+1}\e^{-n})$. 

The contribution of the integer observations to the relative entropy 
is bounded by \eqref{eq-mark-cost} and included in
\eqref{eq-product-kl}. The sampling points are continuous random 
variables, and their probability laws enter the proof through
Lemma~\ref{lem-relative-entropy-cover}, which bounds the cardinality
of separated sets of potentials. A maximal separated set then
provides a finite codebook. This proves the existence of a codebook
with the stated cardinality bound, but gives no algorithm for
constructing it or selecting an approximant for a given potential. 
The resulting bit count refers to the index of an element in a 
codebook shared in advance; it does not include the cost of 
constructing or storing the codebook. 

Such codebooks also have a statistical interpretation.
As discussed in \cite[Introduction]{GW17}, entropy estimates
for convex functions provide complexity bounds for
shape-constrained estimation.
In the toric setting, Theorem~\ref{thm-toric-main} supplies
finite candidate sets for convex-function estimation under
a relative-entropy constraint on the induced transport.

\subsection{Entropy budgets and numerical K\"ahler geometry} 
One possible application of 
Theorem \ref{thm-d1-main} is numerical 
analysis for constant scalar curvature K\"ahler metrics. 
The formula in \cite{Chen00} expresses Mabuchi's K-energy as
relative entropy plus energy terms.
Under an appropriate coercivity hypothesis, control of these
energy terms yields an entropy bound along a normalized
minimizing sequence. 
The compactness in Theorem~\ref{thm-bbegz-entropy}, together
with lower semicontinuity, then gives a minimizer in $\cE^1$.
Further, in the K\"ahler--Einstein setting, the variational approach of
\cite{BBJ21} relates coercivity to a notion of algebraic stability. 
Note  that even a smooth cscK metric need not have small relative entropy with respect to an arbitrarily chosen smooth reference measure.
To apply the result to an iterative method, one must separately establish a uniform entropy budget for the reconstructed potentials. 
Constructing iterative schemes with this uniform entropy control,
together with an implementable codebook construction,
is left for future work.

As a more direct consequence, consider the Monge--Amp\`ere equation
\begin{equation*}
 \MA(u)=\mu,\qquad \Ent(\mu\mid\nu)\leq B,
\end{equation*}
where $\mu$ is a probability measure on $X$.
Since finite entropy implies finite energy,
\cite[Theorem~4.7]{BBGZ13} guarantees a unique solution
$u\in\cE^1(X,\omega)$ normalized by $\sup_Xu=0$. 
If the K\"ahler class is polarized, the upper bound in
Theorem~\ref{thm-d1-main} provides, for each prescribed accuracy
in the $d_1$ metric, a common finite dictionary approximating
these solutions uniformly over all such $\mu$, with an explicit
bound on the dictionary size.

\subsection{Finite-bit approximation of plurisubharmonic potentials}
Our main theorems have the following direct approximation consequence. Smooth regularization is classical (Theorem~\ref{thm-smooth-decreasing}). 
On a polarized manifold, Fubini--Study approximation goes back
to Tian \cite{Tia90}; the full smooth Bergman expansion used below
is given by Zelditch \cite[Theorem~1]{Zel98}. 
Combining the entropy bounds with smooth regularization and Fubini--Study approximation gives the following finite-dictionary estimates.

\begin{cor}
\label{cor-finite-bit-psh}
For every sufficiently large integer $b$ there is a finite set $D_b\subset\cH_\omega$, independent of the input potential, such that
\begin{equation}
 \#D_b\leq2^b,\qquad
 \sup_{u\in\PSH_0(X,\omega)}\min_{v\in D_b}\norm{u-v}_{L^1(\nu)}
 \leq C_Xb^{-1/n}.
 \label{eq-finite-bit-psh}
\end{equation}
This order is optimal among all dictionaries of at most $2^b$ elements, even allowing arbitrary $L^1(\nu)$ outputs. On a polarized manifold the elements of $D_b$ can all be chosen to be normalized Fubini--Study potentials of one common degree, depending on $b$.

Assume polarization for the $d_1$ assertion. For every $B\geq0$ and integer $b\geq C_X(1+B)$ there is a finite dictionary of smooth normalized K\"ahler potentials with
\begin{equation}
 \#D_{B,b}\leq2^b,\qquad
 \sup_{u\in\cK_B}\min_{v\in D_{B,b}}d_1(u,v)
 \leq C_X(1+B)^{(n+1)/n}b^{-1/n}.
 \label{eq-finite-bit-d1}
\end{equation}
This dictionary also admits a common-degree normalized Fubini--Study realization. Its centers are not required to remain in $\cK_B$.
\end{cor}
\begin{proof}
For the first assertion, choose $\e=A b^{-1/n}$ in Theorem~\ref{thm-l1-main}, with $A$ fixed so large that $C_XA^{-n}\leq\log2$. This gives a net of at most $2^b$ centers. Approximate each center within $\e$ by a smooth normalized K\"ahler potential using Theorem~\ref{thm-smooth-decreasing} and \eqref{eq-sup-continuity}. The triangle inequality gives \eqref{eq-finite-bit-psh}. Conversely, an arbitrary external dictionary of error $e>0$ gives an internal cover of radius greater than $2e$ by selecting one input from every nonempty dictionary ball. The lower bound in Theorem~\ref{thm-l1-main} therefore implies $b\log2\geq c_Xe^{-n}$, up to a fixed radius factor, and hence $e\geq c_Xb^{-1/n}$.

For \eqref{eq-finite-bit-d1}, use the dyadic construction itself, rather than restricting the stated entropy theorem to small radii. Equation \eqref{eq-dyadic-count} gives, for each $k_j=2^jk_0$, an internal net of radius $a(1+B)/k_j$ and logarithmic cardinality at most $C_X(1+B)k_j^n$. Choose the largest $j$ for which this last quantity is at most $b\log2$. Such $j$ exists for $b\geq C_X(1+B)$, after increasing the fixed constant. Maximality gives
\begin{equation*}
 k_j\asymp_X\left(\frac{b}{1+B}\right)^{1/n},\qquad
 \frac{1+B}{k_j}\asymp_X(1+B)^{(n+1)/n}b^{-1/n}.
\end{equation*}
Apply the smooth-center version of the construction to $\cK_B^{\cH}$.
Since the resulting cover consists of finitely many closed balls,
it also covers its $d_1$-closure $\cK_B$
by Proposition~\ref{thm-closure-main}. 

It remains to justify the common Fubini--Study degree in both assertions. Let $v_1,\ldots,v_M$ be one of these finite smooth normalized dictionaries and set $h_i=h_0e^{-v_i}$. Choose sections $t_{i,q,\alpha}$ orthonormal for $L^2(h_i^q,V^{-1}\omega_{v_i}^n)$. 
For each fixed $i$, the smooth expansion
\cite[Theorem~1]{Zel98} gives
\begin{equation*}
 q^{-n}\sum_\alpha\abs{t_{i,q,\alpha}}_{h_i^q}^2
 =a_{0,i}+O(q^{-1}),\qquad a_{0,i}>0.
\end{equation*}
The remainder is $O(q^{-1})$ in every fixed $C^r$ norm,
with constants depending on $i$ and $r$.
Consequently the smooth Fubini--Study potential
\begin{equation*}
 w_{i,q}=\frac1q\log\left(q^{-n}\sum_\alpha\abs{t_{i,q,\alpha}}_{h_0^q}^2\right)
 =v_i+\frac1q\log\left(q^{-n}\sum_\alpha\abs{t_{i,q,\alpha}}_{h_i^q}^2\right)
\end{equation*}
converges to $v_i$ in $C^\infty$. Also
\begin{equation*}
 \norm{w_{i,q}-\sup_Xw_{i,q}-v_i}_\infty
 \leq2\norm{w_{i,q}-v_i}_\infty\longrightarrow0.
\end{equation*}
For smooth strictly positive potentials, $d_1(a,b)\leq\norm{a-b}_\infty$: the straight path has velocity $b-a$ and each of its Monge--Amp\`ere measures has mass one. The same uniform bound controls $L^1(\nu)$. Since there are finitely many $i$, one sufficiently large common $q$ approximates every dictionary member to the required accuracy. Increasing $q$ also makes $qL$ very ample. The triangle inequality proves both common-degree assertions without increasing cardinality.
\end{proof}

Locally, if $\omega=\ddc\phi_0$, then $u+\phi_0$ and its dictionary approximants are ordinary plurisubharmonic functions. In the polarized case, writing $t_\alpha=f_\alpha e^{\otimes q}$ makes each local approximant
\begin{equation*}
 \phi_0+w_{i,q}-\sup_Xw_{i,q}
 =\frac1q\log\sum_\alpha\abs{f_\alpha}^2-c_{i,q}.
\end{equation*}
These are local representations of the fixed global normalized class.
 Moreover $b$ counts only the index of an element in a dictionary shared in advance. It does not count construction or storage of the dictionary, precision of section coefficients, or computation of the nearest center. In terms of cardinality $M$, the first optimal rate is $(\log M)^{-1/n}$. 

\begin{rem}
\label{rem-linear-green-approximation}
The proof of Lemma~\ref{lem-elliptic-cover} also provides the linear finite-rank approximation
\begin{equation*}
 A_j u=\overline u+\cG P_j\Delta u,\qquad
 \operatorname{rank} A_j\leq C_X2^{2nj},\qquad
 \norm{u-A_ju}_{L^1}\leq C_X2^{-2j}.
\end{equation*}
The rank bound is \eqref{eq-pj-rank}, and the error follows from \eqref{eq-green-blocks} and \eqref{eq-laplacian-tv}. 
Here $\Delta$ is the positive Riemannian Laplacian used in Section \ref{sec-l1}, and $\cG\Delta u=u-\overline u$.
 
 Thus rank $R$ gives error $O_X(R^{-1/n})$. This linear approximation generally does not preserve $\omega$-plurisubharmonicity. Finite rank, positivity-preserving approximation, and finite-bit encoding are different requirements.
\end{rem}

\begin{cor}
\label{cor-degree-controlled-dictionary}
On a polarized manifold, for $b\geq C_X(1+B)$ the dictionary in \eqref{eq-finite-bit-d1} can instead consist of unnormalized Fubini--Study potentials of common degree $m=k+\ell$, where
\begin{equation}
 k\asymp_X\left(\frac b{1+B}\right)^{1/n},\qquad
 m\leq C_X\left(\frac b{1+B}\right)^{1/n}.
 \label{eq-bit-degree}
\end{equation}
The error and cardinality bounds in \eqref{eq-finite-bit-d1} remain valid.
\end{cor}
\begin{proof}
Take the internal dyadic net $v_1,\ldots,v_M\in\cK_B$ at the index $k=k_j$ chosen in the proof of Corollary~\ref{cor-finite-bit-psh}. Replace each $v_i$ by $\beta_kv_i$. Proposition~\ref{prop-full-approximation} gives, for the net center assigned to $u$,
\begin{equation*}
 d_1(u,\beta_kv_i)\leq d_1(u,v_i)+d_1(v_i,\beta_kv_i)
 \leq(a+C_X)(1+B)/k.
\end{equation*}
The selected dyadic degree gives \eqref{eq-bit-degree}, absorbing the fixed $\ell$ because $k\geq k_0$. There is no change in the number of centers.
\end{proof}

The normalization in Corollary~\ref{cor-degree-controlled-dictionary} matters. Corollary~\ref{cor-finite-bit-psh} supplies normalized common-degree dictionaries, but its finite-center smoothing argument gives no quantitative bound for that degree. Subtracting $\sup_X\beta_ku$ cannot be assumed to preserve the $O((1+B)/k)$ estimate, since \eqref{eq-d1-constants-e} gives
\begin{equation*}
 d_1(u,\beta_ku-\sup_X\beta_ku)
 \geq \abs{\sup_X\beta_ku}-d_1(u,\beta_ku).
\end{equation*}
The following example makes this obstruction explicit.

\begin{ex}
\label{ex-normalization-loss}
On $(\mathbb P^1,\cO(1))$ with reference weight $\phi_0(z)=\log(1+\abs{z}^2)$, there is a fixed smooth normalized semipositive potential $u$ of finite entropy such that
\begin{equation}
 \sup_X\beta_ku=-\frac{\log k}{k}+O(k^{-1}),\qquad
 d_1(u,\beta_ku-\sup_X\beta_ku)\geq\frac{\log k-C}{k}.
 \label{eq-normalization-loss}
\end{equation}
\end{ex}
\begin{proof}
Choose $R>0$ and a smooth nondecreasing cutoff $\chi(r)$, zero for $r\leq R$ and one for $r\geq2R$. Define $F(0)=0$ and $F'(r)=\chi(r)\phi_0'(r)$. Then $(rF'(r))'\geq0$, so $F$ is subharmonic, zero on the inner disk, and equal to $\phi_0-C_0$ outside the outer disk. Thus $u=F-\phi_0$ extends smoothly to $\mathbb P^1$, is nonincreasing in $r$, and has supremum zero at $0$. Its Monge--Amp\`ere density is smooth bounded and nonnegative, hence has finite entropy $B_0$.

For a polynomial section $f$ of degree at most $m=k+\ell$, its Hilbert norm is
\begin{equation*}
 \norm{f}_{H_k(u)}^2=\int_{\mathbb C}\abs{f}^2e^{-kF-\ell\phi_0}d\nu.
\end{equation*}
Let $K_{k,u}(z)=\sup_{\norm{f}_{H_k(u)}\leq1}\abs{f(z)}^2e^{-m\phi_0(z)}$. On $\abs{z}\leq R/2$, the submean inequality on disks of fixed radius inside $\abs{z}<R$ bounds $\abs{f(z)}^2$ by $C\norm{f}_{H_k(u)}^2$, since $F=0$ there and the remaining density is uniformly positive. Therefore $K_{k,u}(z)\leq C$. The constant polynomial has norm at most one, since $F,\phi_0\geq0$ and $\nu(X)=1$, and hence $K_{k,u}(0)\geq1$.

On the complement of $\abs{z}<R/2$, one has $u\leq-\delta$ for some $\delta>0$. In finitely many local frames the smooth weights $kF+\ell\phi_0$ vary by at most a fixed constant on disks of radius $k^{-1}$. The submean inequality on such a disk, whose area is comparable to $k^{-2}$, gives
\begin{equation*}
 K_{k,u}(z)\leq Ck^2e^{ku(z)}.
\end{equation*}
Since $N_k=m+1$ and $\beta_ku=m^{-1}\log(K_{k,u}/N_k)$, this bounds $\beta_ku$ by $-\delta/2$ on the complement for large $k$. Its supremum is therefore attained on the inner disk and lies between $-\log N_k/m$ and $(-\log N_k+\log C)/m$. This proves the first part of \eqref{eq-normalization-loss}. The second follows from Proposition~\ref{prop-full-approximation} for the fixed budget $B_0$, the preceding constant-shift inequality, and $m=k+\ell$.
\end{proof}

The obstruction also occurs in a uniformly entropy-bounded family of strictly positive metrics: take $u_k=(1-k^{-2})u$. Its density is $(1-k^{-2})\MA(u)/\nu+k^{-2}$, uniformly bounded. Since $\norm{u_k-u}_\infty=O(k^{-2})$, comparison in \eqref{eq-gu-hilbert} gives $\norm{\beta_ku_k-\beta_ku}_\infty\leq(k/m)\norm{u_k-u}_\infty$. Thus the same logarithmic supremum term persists. This concerns the specific Bergman map $\beta_k$; it does not preclude another normalized dictionary construction.

The corresponding logarithmic upper bound follows from the classical point-extension method originated from Demailly. We state the additional positivity choice explicitly rather than deducing it from the energy estimate.

\begin{lem}
\label{lem-normalized-bergman}
After increasing the fixed auxiliary integer $\ell$ if necessary, one has, for all $u\in\cK_B$ and $k\geq k_0$,
\begin{equation}
 -\frac{n\log k+C_X}{k+\ell}\leq\sup_X\beta_ku\leq\frac{C_X}{k^2},\qquad
 d_1(u,\beta_ku-\sup_X\beta_ku)
 \leq\frac{C_X(1+B+\log k)}k.
 \label{eq-normalized-bergman}
\end{equation}
\end{lem}
\begin{proof}
Local Ohsawa--Takegoshi extension from a point, \cite[Corollary~13.9 and proof of Theorem~14.2]{Dem11}, followed by the singular weighted $\bar\partial$ estimate \cite[Corollary~5.3]{Dem11}, gives, for a sufficiently positive fixed twist,
\begin{equation}
 K_{k,u}(x):=s(x)^*H_k(u)^{-1}s(x)\geq c_Xe^{ku(x)}
 \quad\text{whenever }u(x)>-\infty.
 \label{eq-point-extension-kernel}
\end{equation}
Cover $X$ by smaller coordinate balls compactly contained in trivializing balls of fixed radii. On a ball containing $x$, the weight $k(\phi_0+u)$ is plurisubharmonic. The local point-extension theorem supplies a holomorphic $f$ with prescribed value, chosen so that
\begin{equation*}
 \abs{f(x)}^2e^{-m\phi_0(x)}=e^{ku(x)},\qquad
 \int\abs{f}^2e^{-k(\phi_0+u)-\ell\phi_0}d\nu\leq C_X.
\end{equation*}
The additional factors involving the fixed weight $\ell\phi_0$ and the smooth volume are uniformly comparable on the finitely many balls.

Choose a cutoff equal to one near $x$ and a global weight $\psi_x\leq C_X$ equal to $n\log\abs{z-x}^2$ there, with $\ddc\psi_x\geq-C_X\omega$. A cutoff of this local logarithm has these bounds uniformly in $x$, since its derivatives are taken on an annulus of fixed radii. Enlarge $\ell$ so that
\begin{equation*}
 \ell\omega+\Ric(\nu)+\ddc\psi_x\geq\omega
 \quad\text{for every }x.
\end{equation*}
The adjoint singular metric with weight $ku+\psi_x$ on $mL\otimes K_X^{-1}$ then has positive curvature, because its remaining term is $k(\omega+\ddc u)\geq0$. Corollary~5.3 of \cite{Dem11} solves $\bar\partial g=\bar\partial(\chi f)$ with
\begin{equation*}
 \int_X\abs{g}_{h_0^m}^2e^{-ku-\psi_x}d\nu
 \leq C_X\int_{\supp d\chi}\abs{f}_{h_0^m}^2e^{-ku}d\nu\leq C_X.
\end{equation*}
The right side is supported away from $x$, where the singular factor and cutoff derivatives are uniformly bounded. Near $x$, $g$ is holomorphic and weighted integrability against $\abs{z-x}^{-2n}$ forces $g(x)=0$; here $u\leq0$ ensures that $e^{-ku}\geq1$. Since $\psi_x$ is uniformly bounded above, the unweighted Hilbert norm of $g$ is bounded as well. Thus $t=\chi f-g$ is a global holomorphic section with $t(x)=f(x)$ and $\norm{t}_{H_k(u)}^2\leq C_X$. The extremal characterization of the evaluation kernel proves \eqref{eq-point-extension-kernel}.

At a maximum point of $u$, one has $u(x)=0$, so \eqref{eq-point-extension-kernel} and $N_k\leq C_Xk^n$ yield $\sup\beta_ku\geq-(n\log k+C_X)/m$. Conversely $u\leq0$ gives $H_k(u)\geq I_N$ and hence $\beta_ku\leq u_I\leq C_X/k^2$ by \eqref{eq-gu-reference-expansion}. The triangle inequality, \eqref{eq-d1-constants-e}, and Proposition~\ref{prop-full-approximation} give the remaining bound in \eqref{eq-normalized-bergman}. Example~\ref{ex-normalization-loss} shows that the logarithm cannot generally be removed for this normalization of $\beta_k$.
\end{proof}

\bibliographystyle{amsplain}
\bibliography{entropy3}

\providecommand{\bysame}{\leavevmode\hbox to3em{\hrulefill}\thinspace}
\providecommand{\MR}{\relax\ifhmode\unskip\space\fi MR }
\providecommand{\MRhref}[2]{%
  \href{http://www.ams.org/mathscinet-getitem?mr=#1}{#2}
}
\providecommand{\href}[2]{#2}
\begin{thebibliography}{10}

\bibitem{Abr00}
Miguel Abreu, \emph{K\"ahler geometry of toric manifolds in symplectic coordinates}, Symplectic and contact topology: interactions and perspectives, Fields Inst. Commun., vol.~35, Amer. Math. Soc., Providence, RI, 2003, arXiv:math/0004122, pp.~1--24.

\bibitem{Bab58}
Konstantin~I. Babenko, \emph{On the entropy of a class of analytic functions}, Nauchn. Dokl. Vyssh. Shkoly Ser. Fiz.-Mat. Nauk (1958), no.~2, 9--16, In Russian.

\bibitem{BN22}
Oscar~F. Bandtlow and St\'ephanie Nivoche, \emph{New solution of a problem of {K}olmogorov on width asymptotics in holomorphic function spaces}, J. Eur. Math. Soc. \textbf{24} (2022), no.~7, 2493--2532.

\bibitem{BB10}
Robert Berman and S{\'e}bastien Boucksom, \emph{Growth of balls of holomorphic sections and energy at equilibrium}, Invent. Math. \textbf{181} (2010), no.~2, 337--394.

\bibitem{BBGZ13}
Robert Berman, S\'{e}bastien Boucksom, Vincent Guedj, and Ahmed Zeriahi, \emph{A variational approach to complex {M}onge--{A}mp\`ere equations}, Publ. Math. Inst. Hautes \'{E}tudes Sci. \textbf{117} (2013), 179--245, Section and proposition locators follow the arXiv version.

\bibitem{BDL17}
Robert Berman, Tam{\'a}s Darvas, and Chinh Lu, \emph{Convexity of the extended {K}-energy and the large time behavior of the weak {Calabi} flow}, Geom. Topol. \textbf{21} (2017), no.~5, 2945--2988 (English).

\bibitem{BF14}
Robert Berman and Gerard Freixas~i Montplet, \emph{An arithmetic {H}ilbert--{S}amuel theorem for singular hermitian line bundles and cusp forms}, Compos. Math. \textbf{150} (2014), no.~10, 1703--1728.

\bibitem{Ber14}
Robert~J. Berman, \emph{Determinantal point processes and fermions on complex manifolds: large deviations and bosonization}, Commun. Math. Phys. \textbf{327} (2014), no.~1, 1--47 (English).

\bibitem{Berm18}
\bysame, \emph{Determinantal point processes and fermions on polarized complex manifolds: bulk universality}, Algebraic and analytic microlocal analysis. AAMA, Evanston, Illinois, USA, May 14--26, 2012 and May 20--24, 2013. Contributions of the workshops, Cham: Springer, 2018, pp.~341--393 (English).

\bibitem{BB13}
Robert~J. Berman and Bo~Berndtsson, \emph{Real {Monge}-{Amp{\`e}re} equations and {K{\"a}hler}-{Ricci} solitons on toric log {Fano} varieties}, Ann. Fac. Sci. Toulouse, Math. (6) \textbf{22} (2013), no.~4, 649--711 (English).

\bibitem{BBEGZ19}
Robert~J. Berman, S\'{e}bastien Boucksom, Philippe Eyssidieux, Vincent Guedj, and Ahmed Zeriahi, \emph{K\"ahler--{E}instein metrics and the {K}\"ahler--{R}icci flow on log {F}ano varieties}, J. Reine Angew. Math. \textbf{751} (2019), 27--89.

\bibitem{BBJ21}
Robert~J. Berman, S{\'e}bastien Boucksom, and Mattias Jonsson, \emph{A variational approach to the {Yau}-{Tian}-{Donaldson} conjecture}, J. Am. Math. Soc. \textbf{34} (2021), no.~3, 605--652 (English).

\bibitem{Ber09}
Bo~Berndtsson, \emph{Probability measures associated to geodesics in the space of {K{\"a}hler} metrics}, Algebraic and analytic microlocal analysis. AAMA, Evanston, Illinois, USA, May 14--26, 2012 and May 20--24, 2013. Contributions of the workshops, Cham: Springer, 2018, pp.~395--419 (English).

\bibitem{BirSol67}
M.~Sh. Birman and M.~Z. Solomyak, \emph{Piecewise-polynomial approximations of functions of the classes {$W_p^\alpha$}}, Math. USSR-Sb. \textbf{2} (1967), no.~3, 295--317.

\bibitem{BK07}
Zbigniew B{\l}ocki and S{\l}awomir Ko{\l}odziej, \emph{On regularization of plurisubharmonic functions on manifolds}, Proc. Amer. Math. Soc. \textbf{135} (2007), no.~7, 2089--2093.

\bibitem{BEGZ10}
S\'{e}bastien Boucksom, Philippe Eyssidieux, Vincent Guedj, and Ahmed Zeriahi, \emph{Monge--{A}mp\`ere equations in big cohomology classes}, Acta Math. \textbf{205} (2010), no.~2, 199--262.

\bibitem{BLY94}
Jean-Pierre Bourguignon, Peter Li, and Shing~Tung Yau, \emph{Upper bound for the first eigenvalue of algebraic submanifolds}, Comment. Math. Helv. \textbf{69} (1994), no.~2, 199--207 (English).

\bibitem{Bro76}
E.~M. Bronshtein, \emph{{$\varepsilon$}-entropy of convex sets and functions}, Siberian Math. J. \textbf{17} (1976), no.~3, 393--398.

\bibitem{Chen00}
Xiuxiong Chen, \emph{On the lower bound of the {Mabuchi} energy and its application}, Int. Math. Res. Not. \textbf{2000} (2000), no.~12, 607--623 (English).

\bibitem{CGSZ19}
Dan Coman, Vincent Guedj, Sibel Sahin, and Ahmed Zeriahi, \emph{Toric pluripotential theory}, Ann. Polon. Math. \textbf{123} (2019), 215--242, Locators follow arXiv:1804.03387.

\bibitem{Dar15}
Tam\'{a}s Darvas, \emph{The {M}abuchi geometry of finite energy classes}, Adv. Math. \textbf{285} (2015), 182--219, Theorem and corollary locators follow arXiv:1409.2072v2.

\bibitem{Dar19}
\bysame, \emph{Geometric pluripotential theory on {K}\"ahler manifolds}, Advances in complex geometry, Contemp. Math., vol. 735, Amer. Math. Soc., Providence, RI, 2019, arXiv:1902.01982, pp.~1--104.

\bibitem{DLR20}
Tam\'{a}s Darvas, Chinh~H. Lu, and Yanir~A. Rubinstein, \emph{Quantization in geometric pluripotential theory}, Comm. Pure Appl. Math. \textbf{73} (2020), no.~5, 1100--1138, Theorem and equation locators follow the arXiv version.

\bibitem{Dem11}
Jean-Pierre Demailly, \emph{Analytic methods in algebraic geometry}, Surveys of Modern Mathematics, vol.~1, International Press, Somerville, MA, 2012, Locators follow the author's July 2011 manuscript.

\bibitem{Dem}
\bysame, \emph{Complex analytic and differential geometry}, 2012, Online monograph.

\bibitem{Dev98}
Ronald~A. DeVore, \emph{Nonlinear approximation}, Acta Numer. \textbf{7} (1998), 51--150.

\bibitem{Don02}
S.~K. Donaldson, \emph{Scalar curvature and stability of toric varieties}, J. Differential Geom. \textbf{62} (2002), no.~2, 289--349.

\bibitem{Don09}
\bysame, \emph{Some numerical results in complex differential geometry}, Pure Appl. Math. Q. \textbf{5} (2009), no.~2, 571--618.

\bibitem{Ero58}
Vladimir~D. Erokhin, \emph{On asymptotics of $\varepsilon$-entropy of analytic functions}, Dokl. Akad. Nauk SSSR \textbf{120} (1958), 949--952, In Russian.

\bibitem{Fin26}
Siarhei Finski, \emph{Metric entropy of the space of holomorphic functions}, 2026, arXiv:2607.18890.

\bibitem{GW17}
Fuchang Gao and Jon~A. Wellner, \emph{Entropy of convex functions on {$\mathbb R^d$}}, Constr. Approx. \textbf{46} (2017), no.~3, 565--592.

\bibitem{Gil52}
Edgar~N. Gilbert, \emph{A comparison of signalling alphabets}, Bell System Tech. J. \textbf{31} (1952), 504--522.

\bibitem{Gri09}
Alexander Grigor'yan, \emph{Heat kernel and analysis on manifolds}, AMS/IP Studies in Advanced Mathematics, vol.~47, Amer. Math. Soc., Providence, RI, 2009.

\bibitem{Guan99}
Daniel Guan, \emph{On modified {Mabuchi} functional and {Mabuchi} moduli space of {K{\"a}hler} metrics on toric bundles}, Math. Res. Lett. \textbf{6} (1999), no.~5-6, 547--555 (English).

\bibitem{Gue14}
Vincent Guedj, \emph{The metric completion of the {R}iemannian space of {K}\"ahler metrics}, 2014, arXiv:1401.7857.

\bibitem{GZ17}
Vincent Guedj and Ahmed Zeriahi, \emph{Degenerate complex {Monge}-{Amp{\`e}re} equations}, EMS Tracts Math., vol.~26, Z{\"u}rich: European Mathematical Society (EMS), 2017 (English).

\bibitem{Gun11}
Adityanand Guntuboyina, \emph{Lower bounds for the minimax risk using {$f$}-divergences, and applications}, IEEE Trans. Inform. Theory \textbf{57} (2011), no.~4, 2386--2399.

\bibitem{HJ13}
Roger~A. Horn and Charles~R. Johnson, \emph{Matrix analysis}, second ed., Cambridge University Press, 2013.

\bibitem{HKPV06}
J.~Ben Hough, Manjunath Krishnapur, Yuval Peres, and B\'{a}lint Vir\'{a}g, \emph{Determinantal processes and independence}, Probab. Surv. \textbf{3} (2006), 206--229.

\bibitem{KT59}
A.~N. Kolmogorov and V.~M. Tikhomirov, \emph{{$\varepsilon$}-entropy and {$\varepsilon$}-capacity of sets in function spaces}, Uspekhi Mat. Nauk \textbf{14} (1959), no.~2(86), 3--86, English translation: Amer. Math. Soc. Transl. Ser. 2, 17 (1961), 277--364.

\bibitem{Mab87}
Toshiki Mabuchi, \emph{Some symplectic geometry on compact {K}\"ahler manifolds. {I}}, Osaka J. Math. \textbf{24} (1987), no.~2, 227--252.

\bibitem{Niv04}
St\'ephanie Nivoche, \emph{Proof of a conjecture of {Z}ahariuta concerning a problem of {K}olmogorov on the $\varepsilon$-entropy}, Invent. Math. \textbf{158} (2004), no.~2, 413--450.

\bibitem{SZ99}
Bernard Shiffman and Steve Zelditch, \emph{Distribution of zeros of random and quantum chaotic sections of positive line bundles}, Comm. Math. Phys. \textbf{200} (1999), 661--683.

\bibitem{Tia90}
Gang Tian, \emph{On a set of polarized {K}\"ahler metrics on algebraic manifolds}, J. Differential Geom. \textbf{32} (1990), no.~1, 99--130.

\bibitem{Var57}
R.~R. Varshamov, \emph{Estimate of the number of signals in error correcting codes}, Dokl. Akad. Nauk SSSR \textbf{117} (1957), 739--741.

\bibitem{Yau78}
Shing-Tung Yau, \emph{On the {R}icci curvature of a compact {K}\"ahler manifold and the complex {M}onge--{A}mp\`ere equation. {I}}, Comm. Pure Appl. Math. \textbf{31} (1978), no.~3, 339--411.

\bibitem{Zak94}
Vyacheslav~P. Zakharyuta, \emph{Spaces of analytic functions and complex potential theory}, Linear Topol. Spaces Complex Anal. \textbf{1} (1994), 74--146.

\bibitem{Zel98}
Steve Zelditch, \emph{Szeg{\H{o}} kernels and a theorem of {T}ian}, Internat. Math. Res. Notices (1998), no.~6, 317--331.

\bibitem{Zha24}
Kewei Zhang, \emph{A quantization proof of the uniform {Y}au--{T}ian--{D}onaldson conjecture}, J. Eur. Math. Soc. \textbf{26} (2024), no.~12, 4763--4778.

\end{thebibliography}
\end{document}